\documentclass[a4paper,reqno,11pt]{amsart}
\usepackage{graphicx} 

\usepackage{lmodern}

\usepackage[english]{babel}
\usepackage{afterpage}
\usepackage{color}
\usepackage{xcolor}
\usepackage{amsmath,amssymb,amsfonts,amsthm,enumerate}
\usepackage{hyperref,mathrsfs,accents}
\usepackage[T1]{fontenc}
\usepackage[active]{srcltx}
\usepackage{epsfig}
\usepackage{cite}
\usepackage{tikz}
\usepackage{tikz-3dplot}
\usepackage[onehalfspacing]{setspace}
\usepackage{subcaption}
\usepackage{pgfplots}
\usepackage{adjustbox}
\usepackage{comment}
\usepackage{relsize}
\usepackage{changepage}
\usetikzlibrary{patterns,backgrounds}

\tdplotsetmaincoords{60}{115}
\pgfplotsset{compat=newest}

\definecolor{almond}{rgb}{0.99, 0.99, 0.99}

\catcode`@=11 \@addtoreset{equation}{section} \catcode`@=12

\newtheorem{theorem}{Theorem}[section]
\newtheorem{lemma}[theorem]{Lemma}
\newtheorem{definition}[theorem]{Definition}
\newtheorem{remark}[theorem]{Remark}
\newtheorem{proposition}[theorem]{Proposition}
\newtheorem{corollary}[theorem]{Corollary}
\newtheorem{example}[theorem]{Example}

\newcommand{\N}{\mathbb{N}}
\newcommand{\R}{\mathbb{R}}

\newcommand{\e}{\varepsilon}
\newcommand{\C}{{\mathfrak C}}

\newcommand{\weakto}{\rightharpoonup}

\newcommand{\ra}{\rightarrow}

\newcommand{\Hau}{{\mathcal H}} 
\newcommand{\Leb}{{\mathcal L}} 
\newcommand{\Id}{\mathrm{Id}} 
\newcommand{\de}{\partial}
\newcommand{\dist}{\mathop{\mathrm{dist}}}
\newcommand{\Hdist}{\mathop{\mathrm{dist}_{\mathcal{H}}}}
\newcommand{\spt}{\mathop{\mathrm{spt}}}

\newcommand{\loc}{\mathop{\mathrm{loc}}}
\renewcommand{\div}{\mathop{\mathrm{div}}}
\newcommand{\difsim}{\Delta}
\newcommand{\ch}{\mathbf{1}}

\newcommand{\restrict}{\mathbin{\vrule height 1.4ex depth 0pt width
		0.13ex\vrule height 0.13ex depth 0pt width 1.3ex}\,}

\newcommand{\cA}{\mathcal{A}}

\newcommand{\cC}{\mathcal{C}}

\newcommand{\cI}{\mathcal{I}}

\newcommand{\cL}{\mathcal{L}}

\renewcommand{\Subset}{\subset\!\subset}

\newcommand{\Lip}{\mathop{\mathrm{Lip}}}
\newcommand{\pr}{P_{\Omega}}

\newcommand{\p}{P}
\newcommand{\dd}{d}

\newcommand{\Tr}{{\mathrm{Tr}}}

\definecolor{grey}{rgb}{.7,.7,.7}
\definecolor{evidGP}{rgb}{0,0,1}
\definecolor{evidG}{rgb}{0,0.5,0}
\definecolor{almond}{rgb}{0.99, 0.99, 0.99}

\author{Gian Paolo Leonardi}
\address{Dipartimento di Matematica, via Sommarive 14, IT-38123 Povo - Trento (Italy)}
\email{gianpaolo.leonardi@unitn.it}

\author{Giacomo Vianello}
\address{Dipartimento di Matematica, via Trieste 63, IT-35121 - Padova (Italy)}
\email{giacomo.vianello@unipd.it}

\thanks{G.P.Leonardi has been partially supported by: PRIN 2017TEXA3H ``Gradient flows, Optimal Transport and Metric Measure Structures''; PRIN 2022PJ9EFL ``Geometric Measure Theory: Structure of Singular Measures, Regularity Theory and Applications in the Calculus of Variations'' (financed by European Union - Next Generation EU, Mission 4, Component 2 - CUP:E53D23005860006); Grant PID2020-118180GB-I00 ``Geometric Variational Problems''. Giacomo Vianello has been supported by GNAMPA (INdAM) Project 2023: ``Esistenza e propriet\`a fini di forme ottime''.}



\title[Free-boundary monotonicity for almost-minimizers]{Free-boundary monotonicity for almost-minimizers of the relative perimeter}

\begin{document}

\begin{abstract}
Let $E \subset \Omega$ be a local almost-minimizer of the relative perimeter in the open set $\Omega\subset \R^{n}$. We prove a free-boundary monotonicity inequality for $E$ at a point $x\in \de\Omega$, under a geometric property called ``visibility'', that $\Omega$ is required to satisfy in a neighborhood of $x$. Incidentally, the visibility property is satisfied by a considerably large class of Lipschitz and possibly non-smooth domains. Then, we prove the existence of the density of the relative perimeter of $E$ at $x$, as well as the fact that any blow-up of $E$ at $x$ is necessarily a perimeter-minimizing cone within the tangent cone to $\Omega$ at $x$.
\end{abstract}

\maketitle
\tableofcontents

\section{Introduction}

Consider an open set $\Omega\subset \R^{n}$ with Lipschitz boundary, and fix $x_{0}\in \de \Omega$. The main goal of this work is to prove a free-boundary monotonicity inequality for a local almost-minimizer $E\subset \Omega$ of the relative perimeter at the boundary point $x_{0}$, under a suitable geometric property of  $\de\Omega$ near $x_{0}$, which we call ``visibility''.

Monotonicity inequalities are key tools in the regularity theory for minimizers and almost-minimizers of the area functional. In the prototypical setting of $E$ being a local perimeter minimizer near $0\in \Omega$ (or even just a critical point of the perimeter functional) and given $0<r_{1}<r_{2}$ such that the ball $B_{r_{2}}$ of radius $r_{2}$ and center $0$ is contained in $\Omega$, it is known that 
\begin{equation} \label{eq:M}
\int_{(B_{r_{2}} \setminus B_{r_{1}})\cap \de^{*}E} \frac{\left\langle x, \nu_{E} \right\rangle^{2}}{|x|^{n+1}} d \Hau^{n-1}(x) \leq \frac{\p(E;B_{r_{2}})}{r_{2}^{n-1}} - \frac{\p(E;B_{r_{1}})}{r_{1}^{n-1}}\, ,
\end{equation}
where $\p(E;B_{r})$ is the perimeter of $E$ in $B_{r}$, $\Hau^{n-1}$ is the Hausdorff $(n-1)$-dimensional measure in $\R^{n}$, $\de^{*}E$ is the reduced boundary of $E$, and $\nu_{E}$ is the weak interior normal defined for $\Hau^{n-1}$-almost every $x\in \de^{*}E$. 

The first, fundamental consequences of an inequality like \eqref{eq:M} are:
\begin{itemize}
\item[(i)] the monotonicity of the renormalized perimeter $r\mapsto \frac{\p(E;B_{r})}{r^{n-1}}$, which turns out to be a non-decreasing function admitting a finite limit as $r\to 0$ (denoted as $\theta_{E}(0)$, the perimeter density of $E$ at $0$);

\item[(ii)] the fact that, if a perimeter minimizer has a constant renormalized perimeter, then the left-hand side of \eqref{eq:M} vanishes for all $0<r_{1}<r_{2}$ and consequently $E$ coincides up to null sets with a cone with vertex at the origin.
\end{itemize} 

These two facts allow the application of the monotonicity formula in several steps of the proof of the $C^{1,\alpha}$-regularity of the reduced boundary $\de^{*}E$, as well as in the analysis of the singular set $\de E\setminus \de^{*}E$. After the pioneering work of De Giorgi \cite{DeGiorgi1961} on the partial regularity of local perimeter minimizers, the internal regularity theory has been successfully extended to (rectifiable) sets of varying topological type \cite{reifenberg1964analyticity}, area-minimizing integral currents \cite{federer1960normal, almgren2000BigRegPaper, DeLellisSpadaro2014regularity-I,DeLellisSpadaro2016regularity-II,DeLellisSpadaro2016regularity-III}, and varifolds \cite{almgren1968existence,allard1972first}. In the codimension-$1$ case, in particular, the dimension of the singular set is sharply estimated as $\le n-8$, after the work of Bombieri-De Giorgi-Giusti \cite{BombieriDeGiorgiGiusti1969} combined with a dimension-reduction argument due to Federer \cite{simon1984lectures}. In these works, monotonicity formulas arise in specific but substantially equivalent forms, through proofs based either on comparison with cones or by testing the first variation of area with suitable radially-symmetric vector fields. 

A key feature of regularity theory is its stability under perturbations that, at small scales, have a higher order of infinitesimality than area. This has led to the extension of the regularity theory to wider classes of almost-minimizers, including for instance minimizers of isoperimetric or prescribed mean curvature problems (see, e.g., \cite{Massari1974,SchoenSimon1981,Bombieri1982,GonzalezMassariTamanini1983,tamanini1982boundaries}).

Coming to the boundary properties of almost-minimizers, we record regularity results proved for integral currents minimizing parametric elliptic integrals and $k$-varifolds with mean curvature in $L^{p}$ for $p>k$, assuming $C^{1,\alpha}$ regularity of their boundary \cite{allard1975boundary,DuzaarSteffen2002,Bourni2016}. In the free-boundary case, regularity results have been proved for area-minimizing currents and varifolds when the domain $\Omega$ is either of class $C^{2}$ \cite{HardtSimon1979,GruterJost1986allard,Gruter1987regularity,Jost1986} or, more recently, a local $C^{2}$-deformation of a wedge-type polyhedral cone \cite{EdelenLi2022}. In these cases, monotonicity properties of the renormalized area are shown by testing with locally constructed, almost-radial vector fields that are smooth and tangent to $\de \Omega$.

The background motivation of this paper is the study of the free-boundary properties and regularity of perimeter almost-minimizers when the boundary of the domain $\Omega$ is Lipschitz but not necessarily smooth. We aim to consider $\Omega$ that might not be locally of class $C^{1,\alpha}$ or coincide with a smooth deformation of a polyhedral wedge (as assumed in \cite{EdelenLi2022}). Carrying over the full regularity program in such a general setting seems particularly challenging. For this reason, in this paper, we focus on the monotonicity property as a preliminary step towards the extension of the free-boundary regularity results mentioned so far. 

\subsection{Description of the main definitions and results}
Let $\Omega \subset \R^n$ be an open set with Lipschitz boundary and $E \subset \Omega$ be a measurable set. In what follows, $E$ will be called a local almost-minimizer of the relative perimeter in $\Omega$ if, for any $x\in \overline{\Omega}$ there exists $r_{x} > 0$ such that, for any $0 < r < r_{x}$ and any measurable $F\subset\Omega$ with $F \difsim E \subset \subset B_{r}(x)$, one has
\begin{equation*}
	\p(E;B_{r}(x) \cap \Omega) \leq \p(F;B_{r}(x) \cap \Omega) + |F \difsim E|^{\frac{n-1}{n}} \psi_{\Omega}(E;x,r) \, ,
\end{equation*}
for a suitable function $\psi_{\Omega}(E;x,r)$ such that $\lim_{r\to 0^{+}}\psi_{\Omega}(E;x,r) = 0$. Of course, when the error term $|F \difsim E|^{\frac{n-1}{n}} \psi_{\Omega}(E;x,r)$ vanishes, we have a local perimeter minimizer. We conveniently introduce the following, classical notation. Given a function $f \in BV_{\loc}(\Omega)$\footnote{Here we mean that $f \in L^{1}_{\loc}(\Omega)$ and has bounded variation in $\Omega \cap A$, for any bounded set $A \subset \R^n$.}, we define the \textit{minimality gap} of $f$ in $A$ as
\begin{equation*} 
\Psi_{\Omega}(f;A) = |Df(\Omega \cap A)| - \inf\big\{ |Dg|(\Omega \cap A) \, : \, g \in BV_{\loc}(\Omega),\ \spt(g-f)\Subset A \big\} 
\end{equation*}
and, when $f = \ch_E$, we set $\Psi_{\Omega}(E;A) := \Psi_{\Omega}(\ch_E;A)$. If $E$ is an almost-minimizer and $x \in \overline{\Omega}$, it is immediate to check that $\Psi_{\Omega}(E;B_r(x)) \lesssim r^{n-1} \psi_\Omega(E;x,r)$. The two main applications of Theorem \ref{thm:MonotonicityInequality} (our general monotonicity inequality) are Corollary \ref{cor:BM} and Theorem \ref{thm:conblowup}. For them, we will need to assume that the minimality gap on balls centered at $0 \in \de \Omega$ decays to $0$ suitably fast as $r \to 0^+$, which is expressed through the following summability property:
\begin{equation} \label{eq:decaymingap}
\int_{0}^{R}\dfrac{\Psi_{\Omega}(E;B_{\rho})}{\rho^n}\, d\rho \simeq  \int_{0}^{R}\dfrac{\psi_{\Omega}(E;0,r)}{\rho}\, d\rho <+\infty\, .
\end{equation}
Several notions of almost-minimality are available in the literature. Among those where the minimality error is in additive form, like ours, we mention the well-known $\Lambda$-minimality \cite{Ambrosio_Introduttivo_00,maggi2012sets}, where the minimality error is estimated by $\Lambda |F \difsim E|$ for some $\Lambda>0$, which is a special case of our definition. We also remark that the definition of almost-minimality given by Tamanini in \cite{Tamanini_quadLecce}\footnote{in Tamanini's definition, the term $|F \difsim E|^{\frac{n-1}{n}}$ is replaced by $r^{n-1}$.} is slightly more general than ours, even though we take direct inspiration from it. The reason why we are forced to consider a slightly stronger almost-minimality is that, in the boundary case, it seems impossible to use the monotonicity inequality to prove lower-density estimates for the perimeter and the volume of $E$ at $x\in \de\Omega$, as it happens for the interior case. On the other hand, a local almost-minimizer in our sense turns out to fulfill volume and perimeter density estimates directly (see Lemma \ref{lem:density}). We point out that these estimates are crucially used in the proof of Theorem \ref{thm:conblowup}. 

The key assumption needed for the proof of our monotonicity formula is the aforementioned visibility property, that we now briefly describe. Let $\Omega \subset \R^n$ be an open set with Lipschitz boundary such that $0 \in \de \Omega$ and in graphical form around $0$ with respect to the variables $x' = (x_1,...,x_{n-1})$, so that $\Omega$ is locally the epigraph of some Lipschitz function of $x'$. We say that $\Omega$ satisfies the visibility property at $0$ if there exist $R > 0$ and a function $v \in C^1([0,R))$ such that:
\begin{itemize}
	\item[(V1)] $v(0)=v'(0)=0$ and $0\le v' \le 2^{-1}$;
	\item[(V2)] The function
	\[
	\gamma_{v}(r) := r^{-1}\sup_{0<s\le r} \sqrt{\frac{v(s)}{s}+v'(s)}
	\]
	is summable on $(0,R)$;
	\item[(V3)] for all $0 < r < R$, the segment joining the point $V_r = (0,\dots,0,-v(r))$ with a point $x$ belonging to $\de \Omega \cap B_{r}$ does not intersect $\Omega$.
\end{itemize}
The visibility property allows us to construct a quasi-conical competitor for a local almost-minimizer of the relative perimeter in $\Omega$, which is a key step in the proof of our monotonicity inequality. It is worth observing that this property guarantees the existence of the tangent cone to $\Omega$ at $0$ (Proposition \ref{prop:tangcone}). 
The visibility property at $0$ is satisfied for instance by cones with vertex at $0$ (with the trivial choice $v \equiv 0$), by $C^{1,\beta}$ open sets and also by convex sets suitably approximated by their tangent cone at $0$; one can also easily construct examples of Lipschitz domains that do not satisfy the property, (see Section \ref{sec:examples} for more details).
Another consequence of the visibility property is the existence of a foliation of a neighborhood of $0$ by spheres with varying centers, which correspond to the level-sets of a function $\phi$ of class $C^1$, i.e., such that $\phi^{-1}(r) = \de B_r(V_r)$. Moreover, by (V1) and (V2) one can show that $\phi$ is $C^1$-close to the function $|x|$ (see Lemma \ref{lem:coarea}). 

In our main result, Theorem \ref{thm:MonotonicityInequality}, we establish the following monotonicity inequality for $f\in BV_{\loc}(\Omega)$, under the visibility property:
\begin{align} \label{eq:MonotIntro}
	& \left( \int_{\Omega \cap \cA_{r_{1},r_{2}}} \phi^{1-n} \left|\left\langle \nu_f, \nabla \phi \right\rangle \right| \, d |Df| \right)^{2} \leq 2\left( \int_{\Omega \cap \cA_{r_{1},r_{2}}} \frac{|\nabla \phi(x)|}{\phi(x)^{n-1}}\, d |Df| \right)\nonumber \\
	& \qquad \qquad \qquad \cdot \bigg[\mu_{f}(r_{2}) - \mu_{f}(r_{1}) + \int_{r_{1}}^{r_{2}} \dfrac{n-1}{\rho^{n}} \, \Psi_{\Omega}(f;B_{\rho}(V_{\rho})) d \rho + G(f;r_1,r_{2}) \bigg] \, .
\end{align}
Here, $0<r_{1}<r_{2}$ are sufficiently small, $\cA_{r_{1},r_{2}} = B_{r_{2}}(V_{r_{2}})\setminus B_{r_{1}}(V_{r_{1}})$, $\nu_f$ represents the Radon-Nikodym derivative of $Df$ with respect to its total variation $|Df|$, 
\[
\mu_f(r) = \frac{|Df|(\Omega \cap B_r(V_r))}{r^{n-1}}\,,
\]
and $G(f;r_{1},r_{2})$ is an error term whose precise expression is given in \eqref{eq:errtermfBV}, which depends on the properties of the boundary of $\Omega$, and goes to $0$ when $r_{2}\to 0$ as soon as
\begin{equation} \label{eq:densbd}
	\mu_f(r) \leq C \, , \qquad \text{for all $r>0$ small enough and for some $C > 0$.}
\end{equation}
Inequality \eqref{eq:MonotIntro} has important consequences when $f = \ch_E$ and $E$ is a local almost-minimizer of the relative perimeter in $\Omega$ satisfying the assumption \eqref{eq:decaymingap}.
Indeed, in this case, the above-mentioned density estimates guarantee that \eqref{eq:densbd} holds. We then infer that the function
\begin{equation*}
	\int_{0}^{r} \dfrac{n-1}{\rho^{n}} \, \Psi_{\Omega}(f;B_{\rho}(V_{\rho})) d \rho + G(f;0,r)
\end{equation*}
is infinitesimal as $r \to 0^+$. Consequently, the density ratio $\mu_f(r)$ is almost-monotone, and hence it admits a finite limit $\theta_E(0)$ as $r \to 0^+$ (see Corollary \ref{cor:BM}). 

Finally, the left-hand side of \eqref{eq:MonotIntro} can be interpreted as an approximate conical deviation term, since it is as small as $E$ is close to being a cone with vertex at $0$ (at least when $\phi(x) = |x|$, i.e. if the visibility property at $0$ is satisfied with $v \equiv 0$). This observation is exploited in the proof of the last Theorem \ref{thm:conblowup}, where we show that any blow-up sequence of an almost-minimizer of the relative perimeter in $\Omega$ admits a subsequence that converges to a minimizing cone in the tangent cone to $\Omega$ at $0$.
%
%
%
%

\section{Preliminaries}
Let $n \in \mathbb{N}$, $n \geq 2$. For $i = 1,...,n$, we denote by $e_{i}$ the $i$-th vector of the standard basis of $\R^{n}$. Given $x$, $y \in \R^{n}$, their Euclidean product is $\langle x, y \rangle := \sum_{i=1}^{n} x_{i} y_{i}$, while the Euclidean norm of $x$ is $|x| := \sqrt{\langle x, x \rangle}$. When needed, we will write $x = (x',x_{n})$, where $x' = (x_{1}, ..., x_{n-1}) \in \R^{n-1}$. We let $B_{r}(x)$ be the open ball of radius $r>0$ centered at $x\in \R^{n}$, and we set $B_{r} := B_{r}(0)$. 
Similarly, we let $B'_{r}(x')$ be the $(n-1)$-dimensional open ball with center $x'$ and radius $r$. We denote by either $\Leb^{n}(E)$ or $|E|$ the Lebesgue measure of $E\subset \R^{n}$, and we set $\omega_{n} := \Leb^{n}(B_{1})$.  
Given a measurable set $A \subset \R^n$, we set
\begin{equation*}
A^{(t)} := \left \{ x \in \R^n \, : \, \lim_{r \ra 0^+} \dfrac{|E \cap B_{r}(x)|}{\omega_n r^n} = t \right \} \, , \qquad \text{for all $0 \leq t \leq 1$ \, .}
\end{equation*}
Given two non-empty, bounded subsets $A$, $B \subset \R^{n}$, we denote by $\Hdist(A,B)$ the Hausdorff distance between $A$ and $B$, that is
\begin{equation*} 
    \Hdist(A,B) = \max \left\{ \sup_{x \in A} \, \dist(x,B) \, , \, \sup_{y \in B} \, \dist(y,A) \right\} \, ,
\end{equation*}
where $\dist(x,Z) := \inf_{z \in Z} |x - z|$. Note that the Hausdorff distance becomes a proper distance function when restricted to compact sets (see \cite{Rockafellar1998}).

Given $\Omega \subset \R^{n}$ open and a vector-valued Radon measure $\mu = (\mu_{1},...,\mu_{p})$ on $\Omega$, we denote by $|\mu|$ its total variation. Let $u = (u_{1},...,u_{p}): \Omega \rightarrow \R^{p}$ be summable with respect to $|\mu|$, then $u \cdot \mu$ denotes the Radon measure defined by
\begin{equation*}
	u \cdot \mu (E) := \int_{E} u \cdot d\mu = \sum_{q = 1}^{p} \int_{E} u_{q}\, d\mu_{q} \, .
\end{equation*}
It can be proved (see \cite[Proposition 1.23]{AmbrosioFuscoPallara2000}) that
\begin{equation*} 
	|u \cdot \mu| = |u| \cdot \mu \, .
\end{equation*}

Given a Lipschitz map $g:\R^{n}\to \R^{p}$, we denote by $\Lip(g)$ its Lipschitz constant, and we denote by $\Lip(\R^{n})$ the space of Lipschitz real-valued functions. We record for future reference the following, elementary fact.
\begin{lemma} \label{lem:uniflem}
Let $\{f_{j}\}_{j\in \N} \subset \Lip(\R^{n})$ be such that $\sup_{j \geq 1} \Lip(f_{j}) < \infty$. Assume further that $f_{j}(x)\rightarrow f(x)$ for all $x\in \R^{n}$. Then $f_{j}\rightarrow f$ locally uniformly on $\R^{n}$.
\end{lemma}

Given $f \in L^{1}_{\loc}(\Omega)$, we define
\begin{equation*} 
	|Df|(\Omega) := \sup \left \{ \int_{\Omega} f(x) \div \phi (x) \, dx \, : \, \phi \in C^{1}_{c}(\Omega; \R^{n}) \, , \, ||\phi||_{L^{\infty}(\Omega)} \leq 1 \right \} \, .
\end{equation*}
We denote by $BV(\Omega)$ the space of functions $f \in L^{1}(\Omega)$ with the property that $|Df|(\Omega) < \infty$. In other words, a function $f\in L^{1}(\Omega)$ belongs to $BV(\Omega)$ if and only if its distributional gradient $Df$ is represented by a $\R^{n}$-valued Radon measure with finite total variation. 

When $E$ is a measurable set, we define the perimeter of $E$ in $\Omega$ as
\begin{equation*}
	\p(E;\Omega) := |D \ch_{E}|(\Omega) \, ,
\end{equation*}
where $\ch_{E}$ is the characteristic function of $E$. We say that $E$ has \emph{finite perimeter} in $\Omega$ provided $\p(E;\Omega) < + \infty$. In the sequel, we will often use the following notation: given another open set $A \subset \R^n$, we define
	\begin{equation*}
		\pr(E;A) := \p(E;A \cap \Omega) \, ,
	\end{equation*}
with the short form $\pr(E) := \pr(E;A)$ whenever $A$ contains $\Omega$. We call $\pr$ the relative perimeter in $\Omega$, and $\pr(E;A)$ the relative perimeter of $E$ in $\Omega$ restricted to $A$. 

For the specific purposes of this paper, it is convenient to define $BV_{\loc}(\Omega)$ as the space of functions $f \in L^1_{\loc}(\Omega)$ such that $f \in BV(\Omega \cap A)$ for all open, bounded sets $A\subset \R^{n}$. We remark that $f \in BV_{\loc}(\Omega)$ implies that $f$ has locally bounded variation in $\Omega$, but the converse does not necessarily hold.

Let now $f \in BV_{\loc}(\Omega)$. As a consequence of \cite[Corollary 5.11]{maggi2012sets} we can consider the density of $Df$ with respect to $|Df|$ defined by
	\begin{equation*}
		\nu_{f}(x) = \lim_{r \rightarrow 0^+} \dfrac{Df(B_{r}(x))}{|Df|(B_{r}(x))} \, , \qquad \text{for $|Df|$-a.e. } x \in \R^{n} \, ,
	\end{equation*}
and satisfying $Df = \nu_{f} \cdot |Df|$ and $|\nu_{f}(x)| = 1$ for $|Df|$-almost every $x \in \Omega$. When $E$ has locally finite perimeter in $\Omega$, we set $\nu_{E} = \nu_{\ch_{E}}$. 

Next we recall some well-known facts concerning the total variation functional, $BV$ functions, and sets of locally finite perimeter. For further details, the reader can consult, e.g., \cite{maggi2012sets,EvansGariepy2015,Giu84book}.
\begin{theorem}[Lower-semicontinuity] \label{thm:lscBV} 
Let $f,f_{j}\in L^1_{\loc}(\Omega)$ for $j\in \N$ be such that $f_{j}\to f$ in $L^1_{\loc}(\Omega)$. Then 
\[
|Df|(\Omega) \leq \liminf_{j} |Df_{j}|(\Omega).
\]
\end{theorem}
	
BV functions can be approximated by smooth functions in a suitable sense.
	\begin{theorem}[Approximation] \label{thm:AnzGiaq}
		Let $\Omega \subset \R^{n}$ be open and $f \in BV(\Omega)$. Then there exists a sequence $\{ f_{j} \}_{j \in \N} \in C^{\infty}(\Omega) \cap BV(\Omega)$ such that
		\begin{equation} \label{eq:strict}
			||f_{j} - f||_{L^{1}(\Omega)} \rightarrow 0 \, , \qquad |Df_{j}|(\Omega) \rightarrow |Df|(\Omega) \, .
		\end{equation}
		Moreover, provided \eqref{eq:strict} holds, we also have $Df_{j} \weakto^{\ast} Df$, i.e., for all $\phi \in C^{\infty}_{c}(\R^{n};\R^{n})$,
		\begin{equation*}
			\int_{\R^{n}} \phi \cdot d Df_{j} \rightarrow \int_{\R^{n}} \phi \cdot d Df \, , \qquad \text{as } j \rightarrow \infty \, .
		\end{equation*}
	\end{theorem}

\begin{remark} \label{rem:strict}
If property \eqref{eq:strict} holds, we say that the sequence $f_{j}$ \emph{strictly converges} to $f$ (see \cite[Definition 3.14]{AmbrosioFuscoPallara2000}). 
Thus Theorem \ref{thm:AnzGiaq} tells us in particular that $C^{\infty}(\Omega) \cap BV(\Omega)$ is dense in $BV(\Omega)$ with respect to the strict convergence.
Moreover, the sequence $\{ f_{j} \}_{j \in \N}$ can be chosen in such a way that the following extra property is satisfied (see \cite[Remark 1.18]{Giu84book})
\begin{equation} \label{eq:AnzGiaqboosted}
	\lim_{\rho \rightarrow 0^{+}} \rho^{-N} \int_{\Omega \cap B_{\rho}(x_0)} |f_{j} - f| dx = 0 \, , \qquad \text{for all $N > 0$, $x_{0} \in \de \Omega$, $j \in \N \, .$}
\end{equation} 
\end{remark}
\begin{theorem}[Compactness]
Let $\Omega \subset \R^{n}$ be an open, bounded set with Lipschitz boundary, let $C > 0$ be a fixed constant, and for $j\in\N$ let $f_{j}\in BV(\Omega)$ be such that
\begin{equation*} 
\|f_{j}\|_{BV(\Omega)} := \|f_{j}\|_{L^{1}(\Omega)} + |Df_{j}|(\Omega) \leq C \, , \qquad \text{for every $j$.}
\end{equation*}
Then there exists $f \in BV(\Omega)$ and a subsequence $f_{j_{k}}$ of $f_{j}$ such that
\begin{equation*}
\|f_{j_{k}} - f\|_{L^{1}(\Omega)} \rightarrow 0 \, .
\end{equation*}
\end{theorem}

\begin{theorem}[Coarea for $BV$ functions] \label{thm:coarea}
		Let $\Omega \subset \R^{n}$ be open and $f \in L^{1}(\Omega)$. For $t \in \R$, let
		\begin{equation*}
			E_{t} = \{ x \in \Omega \, : \, f(x) > t \} \, .
		\end{equation*}
		Then the following statements hold:
		\begin{itemize}
			\item[(i)] if $f \in BV(\Omega)$, then $E_{t}$ has finite perimeter for almost every $t \in \R$, $t \mapsto \p(E_{t};\Omega)$ is measurable, and
			\begin{equation*}
				|Df|(\Omega) = \int_{\R} \p(E_{t};\Omega)\, dt \, ;
			\end{equation*}
			\item[(ii)] conversely, if $f \in L^{1}(\Omega)$, $t \mapsto \p(E_{t};\Omega)$, and
			\begin{equation*}
				\int_{\R} \p(E_{t};\Omega)\, dt < + \infty \, ,
			\end{equation*}
			then $f \in BV(\Omega)$.
		\end{itemize}
\end{theorem}

\begin{theorem}[Coarea for Lipschitz functions] \label{thm:coareaLip}
If $\phi\in \Lip(\R^{n})$, $E\subset \R^{n}$ is a Borel set, and $f:\R^{n}\to [0,\infty)$ is a Borel function, then
\[
\int_{E}f |\nabla \phi| = \int_{\R} \int_{E\cap \{\phi = t\}}f\, d\Hau^{n-1}\, dt\,,
\]
where $\Hau^{n-1}$ denotes the $(n-1)$-dimensional Hausdorff measure in $\R^{n}$.
\end{theorem}

We will also need the following well-known result, which proof can be found in \cite{maggi2012sets}.
\begin{theorem}[Area Formula] \label{thm:AreaForm}
	Let $M \subset \R^n$ be a locally $\Hau^k$-rectifiable set and $f: \R^n \to \R^m$ be a Lipschitz continuous function, with $1 \leq k \leq m$. Then:
	\begin{itemize}
		\item[(i)] for $\Hau^k$-almost all $x \in M$, the restriction of $f$ to $x + T_x M$, where $T_xM$ is the approximated tangent space to $M$ at $x$, is differentiable at $x$ and we denote by $d^M f_x : T_x M \to \R^m$ its differential;
		\item[(ii)] the following identity holds
		\begin{equation} \label{eq:AreaForm}
			\int_{\R^m} \Hau^0(M \cap \{ f = y \}) d \Hau^k(y) = \int_M J^M f d \Hau^k \, ,
		\end{equation}
	    where $J^M f (x) := \sqrt{\det(d^M f^\ast_x \circ d^M f_x)}$ and $d^M f^\ast_x$ denotes the adjoint of $d^M f_x$.
	\end{itemize}
\end{theorem}

We continue by introducing the notion of trace of a $BV$ function on a Lipschitz boundary.
	\begin{theorem} \label{thm:trace}
		Let $\Omega \subset \R^{n}$ be an open set with Lipschitz boundary and denote by $\nu_{\Omega}$ the unit, inner normal vector defined $\Hau^{n-1}$-almost everywhere on $\de \Omega$. Then there exists a unique linear and bounded \emph{inner trace operator}
\begin{equation*}
	\Tr^{+}(\cdot,\de \Omega): BV(\Omega) \rightarrow L^{1}(\de \Omega, \Hau^{n-1})\,,
\end{equation*}
such that, for all $f \in BV(\Omega)$,
\begin{equation} \label{eq:proptrace}
	\lim_{r \ra 0^+} \, \mint{B_{r}(x) \cap \Omega} |\Tr^{+}(f,\de \Omega)(x) - f(y)| \, dy = 0 \, , \qquad \text{for $\Hau^{n-1}$-a.e. $x \in \Omega \, .$}
\end{equation}
Moreover, for all $f \in BV(\Omega)$ and $\phi \in C^{1}_{c}(\R^{n}, \R^{n})$, the following Gauss-Green formula holds:
\begin{equation*} 
	\int_{\Omega} f \, \div \phi \, dx = -(\phi \cdot Df) \Omega - \int_{\de \Omega} \langle \phi, \nu_{\Omega} \rangle \, \Tr^{+}(f,\de \Omega) \, d \Hau^{n-1}\, .
\end{equation*}
\end{theorem}
The function $\Tr^{+}(f,\de \Omega)$ is called the \emph{inner trace} of $f$ on $\de\Omega$. Formula \eqref{eq:proptrace} ensures in particular that
\begin{equation} \label{eq:proptrace2}
		\Tr^{+}(f,\de \Omega)(x) = \lim_{r \ra 0^+} \, \mint{B_{r}(x) \cap \Omega} f(y) \, dy \, , \qquad \text{for $\Hau^{n-1}$-a.e. $x \in \de \Omega \, .$}
\end{equation}
\begin{remark} \label{rem:contrace}
	Under the assumptions of Theorem \ref{thm:trace}, the operator $\Tr^{+}(\cdot,\de \Omega)$ is also continuous with respect to the strict convergence in $\Omega$ (see \cite[Theorem 2.11]{Giu84book} or \cite[Theorem 3.88]{AmbrosioFuscoPallara2000}). Moreover, by \eqref{eq:AnzGiaqboosted} and \eqref{eq:proptrace2}, we have that the sequence $\{f_{j}\}_{j\in \N}$ of Theorem \ref{thm:AnzGiaq} satisfies
\begin{align} \label{eq:traceappr}
			|\Tr^{+}(f_{j},\de \Omega)(x) - \Tr^{+}(f,\de \Omega)(x)| & = \left| \lim_{r \ra 0^+} \, \mint{B_{r}(x) \cap \Omega} f_{j}(y) \, dy - \lim_{r \ra 0^+} \, \mint{B_{r}(x) \cap \Omega} f(y) \, dy \right| \nonumber \\
			& \leq \lim_{r \ra 0^+} \, \mint{B_{r}(x) \cap \Omega} |f_{j}(y) - f(y)| \, dy = 0\nonumber
\end{align}
 for $\Hau^{n-1}$-almost every $x \in \de \Omega$. In other words, $\Tr^{+}(f_{j},\de \Omega) = \Tr^{+}(f,\de \Omega)$ $\Hau^{n-1}$-almost everywhere on $\de\Omega$, for all $j$.
	\end{remark}
	
	Let $\Omega$, $\Omega' \subset \R^{n}$ open sets with Lipschitz boundary, such that $\overline{\Omega} \subset \Omega'$. For any $f \in BV(\Omega')$, we can also consider a trace of $f$ on $\de\Omega$ computed with respect to $\Omega' \setminus \overline{\Omega}$. We thus set 
\[
\Tr^{-}(f, \de \Omega) := \Tr^{+}(f,\de  (\Omega' \setminus \overline{\Omega}))\restrict \de \Omega\,.
\] 
The function $\Tr^{-} (f,\de \Omega) \in L^{1}(\de \Omega; \Hau^{n-1})$ is called the \emph{outer trace} of $f$ on $\de\Omega$. 
\begin{lemma} \label{lem:extBV}
Let $\Omega$, $\Omega' \subset \R^{n}$ be open sets. In addition, let $\Omega$ be bounded and have Lipschitz boundary, and assume that $\Omega \subset \subset \Omega'$. For all $f \in BV(\Omega')$, we have
\begin{equation} \label{eq:extBV}
|Df|(\de\Omega) = \|\Tr^{+}(f, \de \Omega) - \Tr^{-}(f, \de \Omega)\|_{L^{1}(\de \Omega; \Hau^{n-1})} \, .
\end{equation}
\end{lemma}

\begin{lemma} \label{lem:faetrace}
Fix an open set $A \subset \R^{n}$ and a Lipschitz function $\phi:A\to \R$ of class $C^{1}$, such that $|\nabla \phi(x)|>0$ for all $x\in A$. Set $A_{r} = \phi^{-1}(-\infty,r)$. Then for all $f \in BV(A)$, for $\cL^{1}$-almost every $r\in \R$, and $\Hau^{n-1}$-almost every $x \in \de A_{r}\cap A$, we have
\begin{equation}\label{eq:fxequaltrace}
f(x) = \Tr^{+} (f, \de A_{r})(x) = \Tr^{-} (f, \de A_{r})(x) \, .
\end{equation}
\end{lemma}
\begin{proof}
We observe that, for a.e. $r\in \R$,
\begin{equation}\label{eq:ugualtraces}
\Tr^{+} (f, \de A_{r}) = \Tr^{-} (f, \de A_{r}) \, , \qquad \text{$\Hau^{n-1}$-a.e. on $\de A_{r}\cap A$.}
\end{equation}
Indeed, for the proof of \eqref{eq:ugualtraces} we can combine \eqref{eq:extBV} with $|Df|(\de A_{r}\cap A) = 0$ for almost every $r\in \R$, which in turn comes from the fact that $|Df|$ is a finite measure and $\de A_{r}\cap \de A_{s} \cap A = \emptyset$ whenever $r\neq s$. Let $x \in A$ be a Lebesgue point for $f$, then $x \in \de A_{r}$ if and only if $r = \phi(x)$. Thanks to the smoothness of $\de A_{r}\cap A$ (a consequence of the Implicit Function Theorem) we have
\begin{equation*}
	|A_{r} \cap B_{\rho}(x)| = \dfrac{1}{2} \, \rho^{n} + o(\rho^{n}) \, , \qquad \text{as } \rho \ra 0^+ \, ,
\end{equation*}
and consequently
\begin{align*}
	f(x) & = \lim_{\rho \ra 0^+} \, \dfrac{1}{\omega_{n} \, \rho^{n}} \int_{B_{\rho}(x)} f(y) \, dy \\
			& = \lim_{\rho \ra 0^+} \, \dfrac{1}{\omega_{n} \, \rho^{n}} \left( \int_{A_{r} \cap B_{\rho}(x)} f(y) \, dy + \int_{B_{\rho}(x) \setminus A_{r}} f(y) \, dy \right) \\
			& = \lim_{\rho \ra 0^+} \, \dfrac{1}{2} \, \dfrac{1}{|A_{r} \cap B_{\rho}(x)|} \int_{A_{r} \cap B_{\rho}(x)} f(y) \, dy + \dfrac{1}{2} \, \dfrac{1}{|B_{\rho}(x) \setminus A_{r}|} \int_{B_{\rho}(x) \setminus A_{r}} f(y) \, dy \\
			& = \dfrac{1}{2} \, \Tr^{+}(f, \de A_{r})(x) + \dfrac{1}{2} \, \Tr^{-}(f, \de A_{r})(x)\\
			& = \Tr^{\pm}(f, \de A_{r})(x) \, .
\end{align*}
Since the set of Lebesgue points for $f$ coincides with $A$ up to a $\cL^{n}$-negligible set, by Theorem \ref{thm:coareaLip} we obtain that the first equality in \eqref{eq:fxequaltrace} is verified for $\cL^{1}$-almost every $r\in \R$ and $\Hau^{n-1}$-almost every $x \in \de A_{r}\cap A$, which together with \eqref{eq:ugualtraces} concludes the proof.
\end{proof}

\begin{proposition} \label{prop:convBVboost}
Under the assumptions of Lemma \ref{lem:faetrace}, we take $f,f_{j} \in BV(A)$ for $j\in\N$, such that
\begin{equation*}
||f_j - f||_{L^1(A)} \longrightarrow 0 \, , \qquad |D f_j|(A) \longrightarrow |Df|(A) \, .
\end{equation*}
Then, for almost every $0 < r < 1$, we have
\begin{align}
\Tr(f,\de A_r) &:= \Tr^+(f,\de A_r) = \Tr^-(f,\de A_r) \label{eq:coinctrace1} \\ 
\Tr(f_j,\de A_r) &:= \Tr^+(f_j,\de A_r) = \Tr^-(f_j,\de A_r) \, , \qquad \text{for all $j \geq 1 \, ,$} \label{eq:coinctrace2}
\end{align}
and 
\begin{equation} \label{eq:convboost}
|Df_j|(A_r) \longrightarrow |Df|(A_r) \, , \qquad ||\Tr(f_j,\de A_r) - \Tr(f,\de A_r)||_{L^1(\partial A_{r})} \longrightarrow 0 \, ,
\end{equation}
hence in particular $f_{j}$ strictly converges to $f$ on $A_{r}$.
\end{proposition}

\begin{proof}
Thanks to Lemma \ref{lem:faetrace}, the two identities \eqref{eq:coinctrace1} and \eqref{eq:coinctrace2} hold for almost every $0<r<1$. In particular, for such $r$, we deduce that $|Df|(\de A_{r}) = 0$, hence $|Df|(\overline{A_{r}}) = |Df|(A_{r})$. Moreover, by Theorem \ref{thm:lscBV}, we have
\begin{align*}
\liminf_{j \ra \infty} |Df_j|(A_{r}) &\geq |Df|(A_{r})\\ 
		& =|Df|(\overline{A_{r}}) =  |Df|(A) - |Df|(A \setminus \overline{A_{r}}) \\
		& = \lim_{j \ra \infty} |Df_j|(A) - |Df|(A \setminus \overline{A_{r}}) \\
		& \geq \limsup_{j \ra \infty} |Df_j|(A) - |Df_j|(A \setminus \overline{A_{r}}) \\
		& = \limsup_{j \ra \infty} |Df_j|(\overline{A_{r}}) \\
		& \geq \limsup_{j \ra \infty} |Df_j|(A_{r}) \, ,
\end{align*}
which proves that
\begin{equation*}
|Df|(A_{r}) = \lim_{j \ra \infty} |Df_j|(A_{r}) \, .
\end{equation*}
Since $||f_j - f||_{L^1(A)} \rightarrow 0$, we have in particular $||f_j - f||_{L^1(A_{r})} \rightarrow 0$, and thus $\{ f_j \}_{j \geq 1}$ strictly converges to $f$ in $A_{r}$. Finally, \eqref{eq:convboost} holds because, as observed in Remark \ref{rem:contrace}, the inner trace operator is continuous with respect to the strict convergence.
\end{proof}	

Let $E \subset \Omega$ be a set of locally finite perimeter in $\Omega$. We define the \emph{reduced boundary} of $E$, and we denote it by $\de^{\ast} E$, as the set of those points $x \in \Omega$ such that $|D \ch_{E} (B_{r}(x))| > 0$, for all $r > 0$, and there exists a unit vector $\nu_{E}(x)$ such that
\begin{equation*}
\nu_{E}(x) = \lim_{r \ra 0^+} \dfrac{D \ch_{E}(B_{r}(x))}{|D \ch_{E}(B_{r}(x))|} \, .
\end{equation*}
Clearly, we have that $\nu_{E}(x) = \nu_{\ch_{E}}(x)$ for $|D\ch_{E}|$-almost every $x\in \Omega$, therefore the perimeter measure $|D\ch_{E}|$ is concentrated on the reduced boundary $\de^{*}E$. 
\begin{theorem}[De Giorgi-Federer's Structure Theorem] \label{thm:DeGiorgi}
		Let $E \subset \R^{n}$ be a set of locally finite perimeter. Then
		\begin{equation*}
			|D\ch_{E}| = \Hau^{n-1} \restrict \de^{\ast} E \, , \qquad D \ch_{E} = \nu_{E} \cdot \Hau^{n-1} \restrict \de^{\ast} E \, ,
		\end{equation*}
		and the following Gauss-Green Formula holds:
		\begin{equation*}
			\int_{E} \nabla \phi \, dx = \int_{\de^{\ast} E} \phi \nu_{E} \, d \Hau^{n-1} \qquad \forall\, \phi\in C^{1}_{c}(\R^{n})\, .
		\end{equation*}
		Moreover $\de^{\ast} E$ is countably $(n-1)$-rectifiable, i.e. there exist countably many $C^{1}$ hypersurfaces $M_{j} \subset \R^{n}$ and a Borel set $F$ with $\Hau^{n-1}(F) = 0$, such that
		\begin{equation*}
			\de^{\ast} E \subset F \cup \bigcup_{j \geq 1} M_{j} \, .
		\end{equation*}
Finally, $\de^{*}E \subset E^{(1/2)}$ and $\Hau^{n-1}(\R^{n}\setminus(E^{(0)}\cup E^{(1)}\cup \de^{*}E)) = 0$.
\end{theorem}

\subsection{A local extension result}
	From now on, we assume that $\Omega \subset \R^{n}$ is an open set with Lipschitz boundary. We fix $\rho>0$ such that, up to an isometry, there exist a Lipschitz function $\omega: B'_{\rho} \rightarrow \R$ and a constant $m>0$, with $\omega(0) = 0$ and $m > \|\omega\|_{L^{\infty}(B'_{\rho})}$, satisfying the following property: if we set $\cC_{\rho,m} = B'_{\rho} \times (-m,m)$, we have
	\begin{equation} \label{eq:omegagraph}
		\Omega \cap \cC_{\rho,3m} = \{ x=(x',x_{n}) \in \R^{n} \, : \, x' \in B'_{\rho} \, , \, \omega(x') < x_{n} < 3m \} \, .
	\end{equation}
We aim to prove that, under this assumption, any measurable set $E \subset \Omega$ with $\ch_E \in BV_{\loc}(\Omega)$\footnote{We recall that in this paper $f\in BV_{\loc}(\Omega)$ means $f\in BV(A)$ for all $A\subset \Omega$ open and bounded.} can be extended to a locally finite perimeter set $\tilde{E}$ in  $\Omega\cup\cC_{\rho,m}$, in such a way that $\tilde E \cap \Omega = E\cap \Omega$, $P(\tilde E; \partial \Omega\cap \cC_{\rho,m}) = 0$, and $P(\tilde E;S(B)) \le C\, P(E;B)$, for all Borel sets $B\subset \cC_{\rho,m}\setminus \Omega$ and for some constant $C>0$ depending on the dimension $n$ and the function $\omega$. In what follows, we will denote by $T_{x} E$ the approximate tangent space to $\de^\ast E$ at $x$\footnote{The approximate tangent space $T_{x} E$ is given by the orthogonal complement of $\nu_{E}(x)$.}.
    
Set $\cC_{\rho} = B'_{\rho}\times \R$ and define the map $S : \cC_{\rho} \ra \cC_{\rho}$ as
\begin{equation}\label{eq:Ssymmetry}
S(x) := (x', 2 \omega(x') - x_n) \, .
\end{equation}
Note that $S$ satisfies $S^{2}(x) =x$ for all $x$. Moreover, elementary computations show that 
\begin{equation}\label{eq:lipS}
\Lip(S) \le \sqrt{3+6\Lip(\omega)^{2}}\,.
\end{equation}
Given $E\subset \Omega$ measurable with $\ch_E \in BV_{\loc}(\Omega)$, we define $\tilde E \subset \Omega\cup \cC_{\rho}$ as
\begin{equation}\label{eq:tildeE}
\tilde E = E \cup (S_{\rho}(E)\setminus \Omega)\,,
\end{equation}
where $S_{\rho}(E) = S(E\cap \cC_{\rho})$. Clearly, we have $\tilde E\cap \Omega = E\cap \Omega$. 
Further properties of $\tilde E$ are stated in the next lemma.

\begin{lemma} \label{lem:boundS}
Let $E \subset \Omega$ be a measurable set with $\ch_E \in BV_{\loc}(\Omega)$. Then, for almost all $x \in \de^\ast E$, $S$ restricted to $x+T_{x}E$ is differentiable at $x$, and we denote by $d^E S_x: T_x E \ra \R^{n}$ its differential. Moreover, if $\tilde E$ is the set defined in \eqref{eq:tildeE}, we have 
\begin{equation}\label{eq:pEtildezero}
\p(\tilde E;\de\Omega\cap \cC_{\rho,m}) = 0
\end{equation}
and, for all Borel sets $B \subset S(\cC_{\rho,m} \cap \Omega)$, 
\begin{equation} \label{eq:perim}
\p(\tilde E;B) = \p(S_{\rho}(E);B) = \int_{\de^\ast E \cap S(B)} J^E S (x) \, d \Hau^{n-1}(x) \le C\,\p(E;S(B)) \, ,
\end{equation}
where  and $C = \Lip(S)^{n-1}$.
\end{lemma}
\begin{proof}
Owing to Theorem \ref{thm:DeGiorgi}, we know that $\de^\ast E$ is countably $(n-1)$-rectifiable. Then the fact that $S|_{T_x E}$ is differentiable at $\Hau^{n-1}$-almost every $x \in \de^\ast E$ follows immediately from statement (i) of Theorem \ref{thm:AreaForm}.

Now, we prove \eqref{eq:pEtildezero} in the following way. 
Thanks to Lemma \ref{lem:extBV} we only need to check that, for $\Hau^{n-1}$-almost every $x\in \de \Omega\cap \cC_{\rho,m}$, we have
\[
\Tr^{+}(\ch_{\tilde E},\de\Omega)(x) = \Tr^{-}(\ch_{\tilde E},\de\Omega)(x)\,,
\]
that is,
\begin{equation}\label{eq:trequal}
\Tr^{+}(\ch_{E},\de\Omega)(x) = \Tr^{+}(\ch_{S_{\rho}(E)},\de(\R^{n}\setminus \Omega))(x)\,.
\end{equation}
The proof of \eqref{eq:trequal} goes as follows. We employ the characterization of the trace as a limit of averages: for $\Hau^{n-1}$-almost every $x\in \de\Omega$ we have
\[
\Tr^{+}(\ch_{E},\de\Omega)(x) = \lim_{r\to 0^{+}}\frac{|E\cap B_{r}(x)\cap \Omega|}{|B_{r}(x)\cap \Omega|}
\]
and
\[
\Tr^{+}(\ch_{S_{\rho}(E)},\de(\R^{n}\setminus \Omega))(x) = \lim_{r\to 0^{+}}\frac{|S_{\rho}(E)\cap B_{r}(x)\setminus \Omega|}{|B_{r}(x)\setminus \Omega|}\,.
\]
Then we combine this characterization with a consequence of \eqref{eq:proptrace} --- namely, that the trace of a $BV$ characteristic function coincides with a characteristic function $\Hau^{n-1}$-almost everywhere on $\de\Omega$, to infer that we only need to show the equivalence 
\[
\Tr^{+}(\ch_{E},\de\Omega)(x) = 0 \quad\Leftrightarrow\quad \Tr^{+}(\ch_{S_{\rho}(E)},\de(\R^{n}\setminus \Omega))(x) = 0\,.
\]
One of the two required implications (the other can be discussed similarly) is
\[
\Tr^{+}(\ch_{E},\de\Omega)(x) = 0\quad\Rightarrow\quad \Tr^{+}(\ch_{S_{\rho}(E)},\de(\R^{n}\setminus \Omega))(x) = 0\,.
\]
This implication can be restated as 
\begin{equation}\label{eq:EimplSE}
|E\cap B_{r}(x)\cap \Omega| = o(r^{n})\quad\Rightarrow\quad |S_{\rho}(E)\cap B_{r}(x)\setminus \Omega| = o(r^{n})\qquad \text{as $r\to 0^{+}$.}
\end{equation}
Up to taking $r>0$ small enough, we have $B_{r}(x) \subset \cC_{\rho,m}$, hence setting $L = \Lip(S)$ we get
\begin{align*}
|S_{\rho}(E)\cap B_{r}(x)\setminus \Omega| &\le |S\big(E\cap S(B_{r}(x))\cap \Omega\big)|\\
&\le L^{n}|E \cap B_{Lr}(x)\cap \Omega|\\
&= L^{n}o(L^{n}r^{n}) = o(r^{n})\qquad \text{as }r\to 0^{+}\,,
\end{align*}
which proves the implication \eqref{eq:EimplSE} and concludes the proof of \eqref{eq:pEtildezero}.

Finally, for the proof of \eqref{eq:perim}, it is enough to show that
\begin{equation} \label{eq:redbim}
\Hau^{n-1}(\de^\ast S(E) \difsim S(\de^\ast E)) = 0 \, .
\end{equation}
Indeed, if \eqref{eq:redbim} holds, Theorem \ref{thm:DeGiorgi} ensures that
\begin{equation*}
\p(S(E);B) = \Hau^{n-1}(\de^\ast S(E) \cap B) = \Hau^{n-1}(S(\de^\ast E) \cap B) = \Hau^{n-1}(S(\de^\ast E \cap S(B))) \, ,
\end{equation*}
thus \eqref{eq:perim} is an immediate consequence of the Area Formula for rectifiable sets given in \eqref{eq:AreaForm}. Let us demonstrate \eqref{eq:redbim}. Again by Theorem \ref{thm:DeGiorgi}, it suffices to prove that
\begin{equation*} 
S(E)^{(0)} = S(E^{(0)}) \, , \qquad S(E)^{(1)} = S(E^{(1)}) \, .
\end{equation*}
Let us prove the first of the previous identities, as the proof of the other one is obtained by observing that $(\R^{n}\setminus E)^{(0)} = E^{(1)}$. Let $0<\delta <\rho$, and $x \in \Omega \cap (B'_{\rho - \delta} \times \R)$. Set $L=\Lip(S)$ as before, then by construction, for any $0 < r < \delta$, $S(B_{r}(S(x))) \subset B_{L r}(x)$, and thus the Area Formula (see \eqref{eq:AreaForm}) yields
\begin{equation} \label{eq:E(0)}
|S(E) \cap B_{r}(S(x))| = \int_{E \cap B_{r}(S(x))} J S\, dx \leq L^{n} |B_{L r}(x)| \, . 
\end{equation}
Since $r$ is arbitrary and $S^{-1}=S$, it is easy to check that, thanks to \eqref{eq:E(0)}, if $x \in E^{(0)}$ then $S(x) \in S(E)^{(0)}$, which proves the inclusion $S(E^{(0)})\subset S(E)^{(0)}$. The reverse inclusion is proved in a completely analogous way. The proof of the lemma is then achieved thanks to \eqref{eq:lipS}.
\end{proof}

	
\section{Almost minimality}
\label{sec:almost-minimality}
Here we introduce the almost-minimality for the relative perimeter, which is needed in the proofs of the free-boundary monotonicity results shown in Section \ref{sec:BMF}. As a consequence of our definition, we show some properties of the \textit{minimality gap} function, which will be needed later on, and (boundary) density estimates for the volume and the perimeter of an almost-minimizer. 
\begin{definition}[Almost-minimality] \label{def:A-M}
Given $\Omega \subset \R^{n}$ open and $E \subset \Omega$ measurable, we say that $E$ is a \emph{local almost-minimizer} of $\pr$ if, for any $x\in \overline{\Omega}$ there exists $r_{x} > 0$ such that, for any $0 < r < r_{x}$ and any measurable $F\subset\Omega$ with $F \difsim E \subset \subset B_{r}(x)$, one has
\begin{equation}\label{eq:A-M}
\pr(E;B_{r}(x)) \leq \pr(F;B_{r}(x)) + |F \difsim E|^{\frac{n-1}{n}} \psi_{\Omega}(E;x,r) \, ,
\end{equation}
for a suitable function $\psi_{\Omega}(E;x,r)$ such that $\lim_{r\to 0^{+}}\psi_{\Omega}(E;x,r) = 0$. 
If additionally
\[
\psi_{\Omega}(E;r):= \sup_{x\in \overline{\Omega}} \psi_{\Omega}(E;x,r) \rightarrow 0\qquad \text{as }r\to 0^{+}\,,
\] 
then we say that $E$ is an \emph{almost-minimizer} of $\pr$. 
\end{definition}
We note that if $\psi_{\Omega}(E;x,r) = 0$, then $E$ is a minimizer of $\pr$ in $B_{r}(x)$. In particular, if $\psi_{\Omega}(E;x,r) = 0$ for all $x\in \overline{\Omega}$ and $r>0$, then $E$ is a (global) minimizer of $\pr$. We shall later discuss suitable summability properties of the function $r\to \psi_{\Omega}((E;x,r)$, which are required in the proof of the Monotonicity Formula. 

It is worth recalling that, among the various notions of almost-minimality for sets of locally finite perimeter that can be found in the literature, the one expressed by \eqref{eq:A-M} can be understood as a generalization of the well-known $\Lambda$-minimality property \cite{Ambrosio_Introduttivo_00,maggi2012sets}, however, it is slightly more restrictive than the notion considered by Tamanini in \cite{Tamanini_quadLecce} (however, we point out that Tamanini's work is focused on internal regularity theory). For more completeness, we recall that a set $E$ is a \emph{$\Lambda$-minimizer} of $\pr$ if, for any $x \in \overline{\Omega}$ there exists $r_x > 0$ such that, for any $0 < r < r_x$ and any measurable subset $F$ of $\Omega$ with $F \difsim E \Subset B_r(x)$, one has
	\begin{equation*} 
		\pr(E;B_{r}(x)) \leq \pr(F;B_{r}(x)) + \Lambda |F \difsim E| \, .
	\end{equation*}
It is then clear that a $\Lambda$-minimizer satisfies \eqref{eq:A-M} with
	\begin{equation*}
		\psi_{\Omega}(E;x,r) = \Lambda \, r \, \omega_{n}^{1/n} \, .
	\end{equation*}
For better study of almost-minimizers, it is convenient to introduce the notion of \textit{minimality gap}. 	
	\begin{definition} \label{def:mingap}
		Let $\Omega$, $A \subset \R^{n}$ be open sets and $f \in BV_{\loc}(\Omega)$. The \emph{minimality gap} of $f$ in $A$ relative to $\Omega$ is
		\begin{equation*} 
			\Psi_{\Omega}(f;A) = |Df(\Omega \cap A)| - \inf\big\{ |Dg|(\Omega \cap A) \, : \, g \in BV_{\loc}(\Omega),\ \spt(g-f)\Subset A \big\} \, .
		\end{equation*}
	\end{definition}
When $f = \ch_{E}$, for some measurable subset $E \subset \R^{n}$, we write $\Psi_{\Omega}(E;A)$ in place of $\Psi_{\Omega}(\ch_{E};A)$.

\subsection{Some estimates for the minimality gap}
\begin{lemma} \label{lem:eqgaps}
Let $\Omega$, $A \subset \R^{n}$ be open sets and $E \subset \Omega$ be such that $\ch_E \in BV_{\loc}(\Omega)$. Then
\begin{equation*}
\Psi_{\Omega}(E;A) = \pr(E;A) - \inf\big\{ \pr(F;A) \, : \, \ch_{F} \in BV_{\loc}(\Omega) \, , \, F \difsim E \Subset A \big\} \, .
\end{equation*}
\end{lemma}
	\begin{proof}
		Let us define
		\begin{align*}
			& \cI_{1} = \inf \{ |Dg|(\Omega \cap A) \, : \, g \in BV_{\loc}(\Omega) \, , \, \spt(g - \ch_{E}) \Subset A \} \\
			& \cI_{2} = \inf \{ \pr(F;A) \, : \, \ch_{F} \in BV_{\loc}(\Omega) \, , \, F \difsim E \Subset A  \} \, .
		\end{align*}
		It suffices to show that $\cI_{1} = \cI_{2}$. For sure, $\cI_{1} \leq \cI_{2}$ because we can take $g = \ch_{F}$ in the definition of $\cI_{1}$. Fix now $\e > 0$, and let $g \in BV_{\loc}(\Omega)$ be such that $\spt(g - \ch_{E}) \subset \subset A$ and
		\begin{equation} \label{eq:geps}
			|Dg|(\Omega \cap A) \leq \cI_{1} + \e \, . 
		\end{equation}
		For $t \in \R$, let us set
		\begin{equation*}
			G_{t} = \{ x \in \Omega \, : \, g(x) > t \} \, .
		\end{equation*}
		We observe that, for each $0 < t < 1$, $G_{t} \setminus \spt(g - \ch_{E}) = E \setminus \spt(g - \ch_{E})$, and so
		\begin{equation} \label{eq:suppGt}
			G_{t} \difsim E \subset \spt(g - \ch_{E}) \subset \subset A \, .
		\end{equation}
		We can now exploit Theorem \ref{thm:coarea} and \eqref{eq:suppGt} to infer that
		\begin{equation} \label{eq:I1I2}
			|Dg|(\Omega \cap A) = \int_{\R} \pr(G_{t};A) dt \geq \int_{0}^{1} \pr(G_{t};A) dt \geq \cI_{2} \, .
		\end{equation}
		By \eqref{eq:geps} and \eqref{eq:I1I2}, we deduce that $\cI_1 \leq \cI_{2} \leq \cI_{1} + \e$ and then, since $\e>0$ is arbitrary, we get $\cI_{1} = \cI_{2}$ and complete the proof of the lemma.
\end{proof}

\begin{remark}
The minimality gap of a local almost-minimizer $E$ satisfies the estimate 
\begin{equation*} 
\Psi_{\Omega}(E;B_{r}(x)) \leq \omega_{n}^{1-\frac{1}{n}} r^{n-1} \psi_{\Omega}(E;x,r)
\end{equation*}
for all $x \in \overline{\Omega}$, $0 < r < r_{x}$, where $\psi_{\Omega}(E;x,r)$ is the function appearing in \eqref{eq:A-M}. This means that $\Psi_{\Omega}(E;B_{r}(x)) = o(r^{n-1})$ as $r\to 0$. However, we anticipate here that the monotonicity formula proved in Section \ref{sec:BMF} will require a slightly stronger assumption on $\psi_{\Omega}(E;x,r)$, namely that $r^{-1}\psi_{\Omega}(E;x,r)$ is summable on $(0,r_{x})$. We notice that this kind of hypothesis is somehow well-known in the context of regularity theory (see \cite{Tamanini_quadLecce}).
\end{remark}

\begin{lemma}
Let $\Omega, A, A' \subset \R^{n}$ be open sets and $f \in BV_{\loc}(\Omega)$. If $A \Subset A'$, then
\begin{equation*}
\Psi_{\Omega}(f;A) \leq \Psi_{\Omega}(f;A') \, .
\end{equation*}
\end{lemma}
\begin{proof}
Let us fix $\e > 0$ and take $g \in BV_{\loc}(\Omega)$ such that 
\begin{equation*}
\spt(g - f) \Subset A \qquad \text{and}\qquad 
|Df|(\Omega \cap A) - |Dg|(\Omega \cap A) \geq \Psi_{\Omega}(f;A) - \e \, .
\end{equation*}
Of course, $g$ also satisfies also $\spt(g - f) \Subset A'$, hence 
we get
\begin{align*}
\Psi_{\Omega}(f;A') & \geq |Df|(\Omega \cap A') - |D g|(\Omega \cap A') \\
& = |Df|(\Omega \cap A) + |Df|(\Omega\cap (A' \setminus A)) - |D g|(\Omega \cap A) - |D f|(\Omega \cap (A' \setminus A)) \\
			& = |Df|(\Omega \cap A) - |Dg|(\Omega \cap A) \\
			& \geq \Psi_{\Omega}(f;A) - \e \, .
\end{align*}
Since $\e$ is arbitrary, we conclude the proof.
\end{proof}

We now prove two key results. The first one is the lower semicontinuity property of the minimality gap for uniform sequences of local almost-minimizers. An extra, technical difficulty arising in the proof, is that the tangent cone to the domain may not locally contain the dilations of the domain itself, since we are in the boundary case. This requires to suitably extend a competitor from the tangent cone $\Omega_{0}$ to the rescaled domains of the form $t^{-1}\Omega$, $t\to 0^{+}$ via Lemma \ref{lem:boundS}.
The second is an upper bound for the difference between the minimality gaps of two $BV_{\loc}$ functions. 

In what follows, we will say that $\Omega_{j}\to \Omega$ locally in Hausdorff distance if there exist $r_{0},m,L>0$, with $r_{0}<m$, and $L$-Lipschitz functions $\omega_{j},\omega:B'_{r_{0}}\to (-m,m)$ providing local graphical representations of $\de \Omega_{j},\ \de\Omega$, respectively,  as in \eqref{eq:omegagraph}, such that $\omega_{j}\to \omega$ uniformly in $B'_{r_{0}}$. 
\begin{lemma}\label{lem:gaplsc}
For $j\in \N$ we let $\Omega_{j},\Omega\subset \R^{n}$ be open sets with uniformly Lipschitz boundary, such that $0\in \de\Omega$ and $\Omega_{j}\to \Omega$ locally in Hausdorff distance. Let $E_{j}, E$ be sets of locally finite perimeter, such that $E_{j}\subset \Omega_{j},\ E\subset \Omega$, and $E_{j}\to E$ in $L^{1}(B_{r_{0}})$. Finally, we assume that $E_{j}$ satisfies \eqref{eq:A-M} for all $0<r<r_{0}$ and $x=0$, and that moreover we have
\[
\lim_{r\to 0^{+}}\, \sup_{j} \psi_{\Omega_{j}}(E_{j};0,r) = 0\,.
\]
Then, $E$ satisfies \eqref{eq:A-M} for all $0<r<r_{0}$, with $\psi_{\Omega}(E;0,r) = \sup_{j} \psi_{\Omega_{j}}(E_{j};0,r)$, and 
\begin{equation}\label{LSCMinGap}
\liminf_{j} \Psi_{\Omega_{j}}(E_{j};B_{r_{0}}) \ge \Psi_{\Omega}(E;B_{r_{0}})\,.
\end{equation}
Moreover, if $\Psi_{\Omega_{j}}(E_{j};B_{r_{0}}) \to 0$ as $j\to\infty$, then for almost all $0<r<r_{0}$ we have
\begin{equation}\label{eq:perimcont}
\lim_{j \ra \infty} \p_{\Omega_{j}}(E_{j}; B_{r}) = \pr(E;B_{r}) \, , \qquad \text{for almost all $r > 0 \, .$}
\end{equation}
\end{lemma}
\begin{proof}
Let us fix $\e > 0$. Let $F \subset \Omega$ be such that $F \difsim E \subset \subset B_{r_{0}}$ and
\begin{equation}\label{eq:PsiFcontroE}
 \Psi_{\Omega}(E;B_{r_{0}}) \le \pr(E;B_{r_{0}}) - \pr(F;B_{r_{0}}) + \e\, .
\end{equation}
By Lemma \ref{lem:faetrace}, for all $j \geq 1$, for almost every $0 < r < r_{0}$ and $\Hau^{n-1}$-a.e. $x\in \de B_{r}$, we have 
\begin{equation}\label{eq:equaltraces}
\ch_{E_{j}}(x) = \Tr^{\pm}(E_{j},\de B_{r})(x)\quad \text{and} \quad  \ch_{E}(x) = \Tr^{\pm}(E,\de B_{r})(x)\,,
\end{equation}
where $\Tr^{\pm}(A,\de B_{r}) := \Tr^{\pm}(\ch_A, \de B_{r})$. By the $L^{1}_{\loc}$-convergence of $E_{j}$ to $E$, we can choose $r<r_{0}$ with the above property and the additional
\begin{equation}\label{eq:perrrzero}
\pr(E;B_{r}) \ge \pr(E;B_{r_{0}}) - \e\,,
\end{equation}
then take $j_{\e}$ large enough, such that 
\begin{align} \label{eq:difcompr}
&F \difsim E \subset \subset B_{r}\,,\\\label{eq:tracesmall}
&\int_{\de B_{r}} |\ch_{E_{j}} - \ch_{E}| \dd \Hau^{n-1} <\e\qquad \text{for $j\ge j_{\e}$} \, .
\end{align}
Let us fix $\delta>0$, define
\[
U_\delta := \big\{x=(x',x_{n})\in \cC_{r,m}:\ \omega(x')-\delta < x_{n}\le \omega(x')\big\}\,,
\]
and assume $\delta$ so small that 
\begin{equation}\label{eq:Haudeltaeps}
\Hau^{n-1}(\de B_{r}\cap U_{\delta})<\e
\end{equation}
and
\begin{equation}\label{eq:peromegadelta}
\pr(E;B_{r}) \le \p(E;B_{r}\cap(\Omega + \delta e_{n})) + \e\,.
\end{equation}
Owing to the uniform convergence of $\omega_{j}$ to $\omega$, we can select $j_\delta \geq 1$ such that $\|\omega_{j}-\omega\|_{\infty} <\delta$ for all $j\ge j_{\delta}$, then for those $j$ we define
\begin{equation*}
A_j := \Omega_{j} \cap \Omega \cap B_r \, , \qquad B_j := (\Omega_{j} \setminus \Omega) \cap B_r \,,
\end{equation*}
and observe that $B_j \subset U_\delta$. Now we set
\begin{equation} \label{eq:Fj}
F_{j} := (\tilde{F} \cap \Omega_{j} \cap B_r) \cup \big(E_{j} \cap (B_{r_{0}} \setminus B_{r})\big) \,,
\end{equation}
where $\tilde{F}=F\cup (S(F)\setminus \Omega)$ and $S$ is the symmetry through $\de \Omega$ defined in \eqref{eq:Ssymmetry} (with $\rho = r_{0}$). Thanks to \eqref{eq:difcompr}, we have $F_{j}\subset \Omega_{j}$ and $F_{j} \difsim E_{j} \subset \subset B_{r_{0}}$, which means that $F_{j}$ is a competitor for $E_{j}$ in the definition of $\Psi_{\Omega_{j}}(E_{j},B_{r_{0}})$. Moreover, by \eqref{eq:Fj}, \eqref{eq:tracesmall}, and \eqref{eq:Haudeltaeps}, we have
\begin{align} \nonumber
\p_{\Omega_{j}}(F_{j};B_{r_{0}}) & \le \p_{\Omega_{j}}(F_{j};B_{r}) + \p_{\Omega_{j}}(E_{j};B_{r_{0}} \setminus \overline{B_{r}}) \\
& \qquad + \int_{\de B_{r} \cap \Omega_{j}\cap \Omega} |\ch_{E_{j}} - \ch_{E}| d \Hau^{n-1} + \Hau^{n-1}(\de B_{r}\cap U_{\delta}) \nonumber \\
\label{eq:pF_j}
&\le \p_{\Omega_{j}}(F_{j};B_{r}) + \p_{\Omega_{j}}(E_{j};B_{r_{0}} \setminus \overline{B_{r}}) + 2\e\,.
\end{align}
Let us compute $\p(F_{j};B_{r} \cap \Omega_{j})$. By Lemma \ref{lem:boundS}, $\p(\tilde F;\de \Omega\cap \cC_{r_{0},m})=0$, hence
	    \begin{align} \label{eq:pF_jr}
	    	\p_{\Omega_{j}}(F_j;B_r) & = \p(F;A_j) + \p(S(F);B_j)\, .
	    \end{align}
        Again by Lemma \ref{lem:boundS}, up to possibly taking a smaller $\delta$, we have
\begin{equation}\label{eq:stimaFtildeBj}
\p(S(F);B_j) = \p(\tilde{F};B_{j}) \le \p(\tilde{F};U_\delta) \leq \e\,.
\end{equation}
Putting together \eqref{eq:pF_j}, \eqref{eq:pF_jr}, and \eqref{eq:stimaFtildeBj}, and taking into account \eqref{eq:equaltraces}, we obtain
\begin{align} \label{eq:EstMinGap1}
\Psi_{\Omega_{j}}(E_{j};B_{r_{0}}) & \geq \p_{\Omega_{j}}(E_{j};B_{r_{0}}) - \p_{\Omega_{j}}(F_{j};B_{r_{0}}) \\
			& \ge \p_{\Omega_{j}}(E_{j};B_{r}) - \p(F; A_{j}) -3\e\nonumber \\
			& \geq \p_{\Omega_{j}}(E_{j};B_{r}) - \pr(F;B_{r_{0}}) -3\e\, . \nonumber
		\end{align}
Now, since for all $j\ge j_{\e}$ we have $\|\omega_{j}-\omega\|_{\infty}<\delta$, we infer that
\[
B_{r}\cap (\Omega+\delta e_{n}) \subset B_{r}\cap \Omega_{j}\,.
\]
This inclusion combined with \eqref{eq:EstMinGap1} and \eqref{eq:peromegadelta} gives
\begin{align} \label{eq:EstMinGap2}
\Psi_{\Omega_{j}}(E_{j};B_{r_{0}}) & \ge \p(E_{j};B_{r} \cap (\Omega+\delta e_{n})) - \p(F;B_{r_{0}} \cap \Omega) -3\e\,,
\end{align}
so that taking the liminf as $j\to\infty$ in \eqref{eq:EstMinGap2}, and using the lower semicontinuity of the perimeter, \eqref{eq:peromegadelta}, \eqref{eq:perrrzero}, and \eqref{eq:PsiFcontroE}, we find
\begin{align*}
\liminf_{j} \Psi_{\Omega_{j}}(E_{j};B_{r_{0}}) &\ge \liminf_{j} \p(E_{j};B_{r} \cap (\Omega+\delta e_{n})) - \pr(F;B_{r_{0}}) -3\e\\
&\ge \p(E;B_{r} \cap (\Omega+\delta e_{n})) - \pr(F;B_{r_{0}}) -3\e\\
&\ge \pr(E;B_{r}) - \pr(F;B_{r_{0}}) -4\e\\
&\ge \pr(E;B_{r_{0}}) - \pr(F;B_{r_{0}}) -5\e\\
&\ge \Psi_{\Omega}(E;B_{r_{0}}) - 6\e\,.
\end{align*}
Then, the arbitrary choice of $\e$ implies \eqref{LSCMinGap}. Now, the fact that $E$ satisfies \eqref{eq:A-M} with $\psi_{\Omega}(E;0,r)$ as in the statement can be proved with the same argument used to show \eqref{LSCMinGap}, also taking into account that $|E_{j}\difsim F_{j}|\to |E\difsim F|$ as $j\to\infty$. Finally, to prove \eqref{eq:perimcont} we consider the sequence $F_{j}$ defined as before, but now with $F=E$. Choosing $\e>0$ arbitrarily, for $j$ large enough we obtain as before
\begin{align*}
\Psi_{\Omega_{j}}(E_{j};B_{r_{0}}) &\ge \p_{\Omega_{j}}(E_{j};B_{r_{0}}) - \p_{\Omega_{j}}(F_{j};B_{r_{0}})\\
&\ge \p_{\Omega_{j}}(E_{j};B_{r}) - \pr(E;B_{r}) - 3\e\,,
\end{align*}
which gives the desired conclusion.
\end{proof}

\begin{lemma} \label{lem:upbmingap}
Let $\Omega, A \subset \R^n$ be open sets, with $\de\Omega$ Lipschitz and $A$ bounded, of class $C^{2}$, and such that $\Hau^{n-1}(\de A\cap \de \Omega) = 0$. Let $f,g\in BV_{\loc}(\Omega)$, then
\begin{align} \label{eq:finalpsi}
|\Psi_\Omega(f,A) - \Psi_\Omega(g,A)| & \leq \Big||Df|(\Omega \cap A) - |Dg|(\Omega \cap A)\Big| \\\nonumber
&\qquad + \|\Tr^+(f_{0}, \de A) - \Tr^+(g_{0},\de A)\|_{L^1(\de A)} \,,
\end{align}
where $f_{0}, g_{0}$ denote the zero-extensions of $f$ and $g$ respectively on $A\setminus \Omega$.
\end{lemma}
\begin{proof}
Given $\e>0$, there exists $h\in BV_{\loc}(\Omega)$ with $\spt(h-f)\Subset A$, such that
\begin{align}\label{eq:stimaPsi1}
\Psi_{\Omega}(f,A) &\le |Df|(A\cap \Omega) - |Dh|(A\cap \Omega) + \e\\\nonumber
&\le \big||Df|(A\cap \Omega) - |Dg|(A\cap \Omega)\big| + \Psi_{\Omega}(g,A) + |D\tilde{h}|(A\cap \Omega) - |Dh|(A\cap \Omega) + \e\,,
\end{align}
where $\tilde{h}\in BV_{\loc}(\Omega)$ will be suitably chosen, so that in particular $\spt(\tilde{h}-g)\Subset A$. For the definition of $\tilde{h}$, we claim that it is possible to construct a sequence $A^{(k)}$ of inner parallel sets of $A$ that converge to $A$, for which $|Df_{0}|(\de A^{(k)}) = |Dg_{0}|(\de A^{(k)}) = 0$ and, moreover, 
\begin{equation}\label{eq:limtrace}
\lim_{k} \int_{\de A^{(k)}} |\Tr^{+}(f_{0}-g_{0},\de A^{(k)})|\, d\Hau^{n-1} = \int_{\de A} |\Tr^{+}(f_{0}-g_{0},\de A)|\, d\Hau^{n-1}\,.
\end{equation}
For the proof of \eqref{eq:limtrace} we argue as follows. Since $A$ is of class $C^{2}$, there exists $\delta>0$ such that, for all $0<t<\delta$, the map $\zeta_{t}(x) = x+t\nu_{A}(x)$ is a diffeomorphism of class $C^{1}$ between $\de A$ and the boundary $\de A_{t}$ of the inner parallel set $A_{t} = \{x\in A:\ \dist(x,\de A)>t\}$. 
Now, we consider two sequences $f_{0,j}, g_{0,j}$ of smooth approximations of $f_{0},g_{0}$ on $A$, with traces $\Tr^{+}(f_{0,j},\de A) = \Tr^{+}(f_{0},\de A)$ and $\Tr^{+}(g_{0,j},\de A) = \Tr^{+}(g_{0},\de A)$, respectively (see Remark \ref{rem:contrace}). By inspecting the proof of Anzellotti-Giaquinta's approximation theorem, it is not restrictive to ask that the sequences $f_{0,j}, g_{0,j}$ also satisfy
\begin{align}\label{eq:precoarea1}
\int_{A\setminus A_{1/k}} (|f_{0,j} - f_{0}|+|g_{0,j} - g_{0}|)\, dx \le \frac{1}{k^{2}}\,,
\end{align}
for all $j$ and for $k> \delta^{-1}$.
We note that the tangential Jacobian of $\zeta_{t}$ satisfies $J\zeta_{t}(x) = 1 + O(t)$, hence the area formula gives
\begin{equation}\label{eq:tracenorms}
\int_{\de A_{t}} |f_{0,j}(y) - g_{0,j}(y)|\, d\Hau^{n-1}(y) = (1+ O(t))\int_{\de A}|f_{0,j}(\zeta_{t}(x)) - g_{0,j}(\zeta_{t}(x))|\, d\Hau^{n-1}(x)\,.
\end{equation}
As $t\to 0^{+}$ we have $\zeta_{t}(x) \to x$ uniformly. Therefore, by following the same Cauchy-sequence argument as in the classical construction of the trace (see, e.g., \cite{EvansGariepy2015}), the compositions $f_{0,j}(\zeta_{t}(x))$ and $g_{0,j}(\zeta_{t}(x))$ can be shown to converge in $L^{1}(\de A)$ to some limits $\hat{f}_{0,j}$ and $\hat{g}_{0,j}$, respectively. Hence \eqref{eq:tracenorms} implies
\begin{equation}\label{eq:tracenorms2}
\lim_{t\to 0^{+}} \int_{\de A_{t}} |f_{0,j}(y) - g_{0,j}(y)|\, d\Hau^{n-1}(y) = 
\int_{\de A}|\hat{f}_{0,j}(x) - \hat{g}_{0,j}(x)|\, d\Hau^{n-1}(x)\,.
\end{equation}
At the same time, if we choose a vector field $\xi\in C^{1}(\R^{n};\R^{n})$ and set either $u_{j}=f_{0,j}$ or $u_{j}=g_{0,j}$, by Gauss-Green Theorem we obtain 
\begin{align*}
\int_{A_{t}} (u_{j}\div \xi + \nabla u_{j}\cdot \xi)\, dx &= -\int_{\de A_{t}} u_{j}\,\xi\cdot \nu_{A_{t}}\, d\Hau^{n-1}\\
&= -(1+O(t))\int_{\de A} u_{j}(\zeta_{t}(x))\,\xi(\zeta_{t}(x))\cdot \nu_{A}(x)\, d\Hau^{n-1}(x)\,, 
\end{align*}
hence taking the limit as $t\to 0^{+}$ gives
\begin{align*}
\int_{A} (u_{j}\div \xi + \nabla u_{j}\cdot \xi)\, dx &= -\int_{\de A} \hat{u}_{j}\,\xi\cdot \nu_{A}\, d\Hau^{n-1}\,, 
\end{align*}
which means that $\hat{f}_{0,j}$ and $\hat{g}_{0,j}$ coincide, respectively, with $\Tr^{+}(f_{0},\de A)$ and $\Tr^{+}(g_{0},\de A)$ up to $\Hau^{n-1}$-null sets and for all $j$, by the uniqueness of the trace. We can thus rewrite \eqref{eq:tracenorms2} as
\begin{equation}\label{eq:tracenorms3}
\lim_{t\to 0^{+}} \int_{\de A_{t}} |f_{0,j}(y) - g_{0,j}(y)|\, d\Hau^{n-1}(y) = 
\int_{\de A}|\Tr^{+}(f_{0}-g_{0},\de A)(x)|\, d\Hau^{n-1}(x)\,.
\end{equation}
To get \eqref{eq:limtrace} from \eqref{eq:tracenorms3}, we must choose $A^{(k)}$ appropriately. To this aim, we apply the coarea formula to the integral in \eqref{eq:precoarea1} and average the resulting inequality, deducing the existence of $0<t_{k}<1/k$ such that for all $j$
\begin{equation}\label{eq:coarea12}
\int_{\de A^{(k)}} (|f_{0,j} - \Tr(f_{0},\de A^{(k)})| + |g_{0,j} - \Tr(g_{0},\de A^{(k)})|) \, d\Hau^{n-1} \le \frac{1}{k}\,,
\end{equation}
where we have set $A^{(k)} = A_{t_{k}}$. By the triangle inequality and \eqref{eq:coarea12} we obtain
\begin{align*}
\int_{\de A^{(k)}} |\Tr^{+}(f_{0}-g_{0},\de A^{(k)})|\, d\Hau^{n-1} &\le \int_{\de A^{(k)}} |f_{0,j} - g_{0,j}|\, d\Hau^{n-1} + \frac{1}{k}\,,
\end{align*}
which gives \eqref{eq:limtrace} at once from \eqref{eq:tracenorms3}.

Now we observe that $|Df_{0}|(\de A^{(k)}) = |Dg_{0}|(\de A^{(k)}) = 0$ because the inner and outer traces of $f_{0}$ and $g_{0}$ on $\de A^{(k)}$ coincide, hence we can define 
\[
\tilde{h} = h\ch_{A^{(k)}} + g\ch_{A\setminus A^{(k)}}\,.
\] 
Note that $\spt(\tilde{h}-g)\Subset A$ and, if $k$ is large enough, $\spt(h-f)\Subset A^{(k)}$, so that
\begin{align*}
|D\tilde{h}|(A\cap \Omega) &\le |Dh|(A^{(k)}\cap\Omega) + |Dg_{0}|(A\setminus \overline{A^{(k)}}) + \int_{\de A^{(k)}} |\Tr(f_{0}-g_{0},\de A^{(k)})|\, d\Hau^{n-1}\,.
\end{align*}
By choosing $k$ large enough we obtain $|Dg_{0}|(A\setminus A^{(k)}) <\e$ and thus
\begin{equation}\label{eq:stimaDelta1}
|D\tilde{h}|(A\cap \Omega) - |Dh|(A\cap\Omega) \le \int_{\de A} |\Tr^{+}(f_{0}-g_{0},\de A)|\, d\Hau^{n-1} + \e.
\end{equation}
By combining \eqref{eq:stimaPsi1} and \eqref{eq:stimaDelta1}, we get
\begin{align*}
\Psi_{\Omega}(f,A) &\le \big||Df|(A\cap \Omega) - |Dg|(A\cap \Omega)\big| + \Psi_{\Omega}(g,A) + \int_{\de A} |\Tr^{+}(f_{0}-g_{0},\de A)|\, d\Hau^{n-1} + 2\e\,.
\end{align*}
Since $\e$ is arbitrary, we conclude
\begin{align*}
\Psi_{\Omega}(f,A) - \Psi_{\Omega}(g,A) &\le \big||Df|(A\cap \Omega) - |Dg|(A\cap \Omega)\big| + \int_{\de A} |\Tr^{+}(f_{0}-g_{0},\de A)|\, d\Hau^{n-1}
\end{align*}
and, by exchanging the role of $f$ and $g$ in the argument above, we obtain \eqref{eq:finalpsi}.
\end{proof}

\subsection{Density estimates} In this subsection we establish perimeter and volume density estimates for almost-minimizers at a point either in $\Omega$ or on $\de \Omega$.
\begin{lemma}\label{lem:density}
Let $\Omega\subset \R^{n}$ be an open set with a Lipschitz boundary. Let $E$ be an almost-minimizer in $\Omega$, and let $x\in \overline{\Omega}$. Assume that $\pr(E;B_{r}(x)) > 0$ for all $r > 0$. Then, there exist constants $C \ge 1,\ r_{0} > 0$, depending on $\Omega$, $E$ and $x$, such that
		\begin{align}
			\label{eq:densityper}
			&C^{-1}r^{n-1} \le \pr(E;B_{r}(x)) \le Cr^{n-1}\\
			\label{eq:densityvol}
			&\min\big(|E\cap B_{r}(x)\cap \Omega|,\ |(B_{r}(x)\cap \Omega)\setminus E|\big) \ge C^{-1}r^{n}\,,
		\end{align}
		for all $0 < r < r_{0}$.
\end{lemma}
\begin{proof}
		Up to a translation, we assume that $x=0$. We start proving \eqref{eq:densityvol}. Given $r>0$ we set
		\[
		m(r) := |B_{r} \cap \Omega\cap E| \, , \quad \mu(r) := |B_{r} \cap \Omega \setminus E| \,.
		\]
		Both $m$ and $\mu$ are non-decreasing, thus differentiable for almost all $r > 0$. By \cite[Example 13.3]{maggi2012sets}, for almost all $r > 0$, we have
		\[
		m'(r) = \Hau^{n-1}(E \cap \de B_{r} \cap \Omega) \, , \quad \mu'(r) = \Hau^{n-1}(\de B_{r} \cap \Omega \setminus E) \, .
		\]
Now, up to an isometry, and for $r_{0}$ sufficiently small and $L\ge 1$ sufficiently large, if we set $Q_{L,r_{0}} = B'_{r_{0}}\times (-Lr_{0},Lr_{0})$ and $\Omega_{L,r_{0}} := \Omega\cap Q_{L,r_{0}}$, we have that $\Omega_{L,r_{0}}$ is connected and has a Lipschitz boundary, hence it supports a relative isoperimetric inequality of the form
\begin{equation} \label{eq:relisop0}
\p(F;\Omega_{L,r_{0}}) \geq C_{0} \min \{|F\cap \Omega_{L,r_{0}}|, |\Omega_{L,r_{0}}\setminus F| \}^{\frac{n-1}{n}} \,.
\end{equation}
From \eqref{eq:relisop0} and the fact that $B_{r}\Subset Q_{L,r_{0}}$ for all $0<r<r_{0}$, we infer that 
\begin{equation} \label{eq:relisop}
m'(r) + \pr(E;B_{r}) = \p(E\cap B_{r};\Omega_{L,r_{0}}) \geq C_{0} \min \{ m(r), \mu(r) \}^{\frac{n-1}{n}} \,.
\end{equation}
and
\begin{equation} \label{eq:relisop-bis}
\mu'(r) + \pr(E;B_{r}) = \p(B_{r}\setminus E;\Omega_{L,r_{0}}) \geq C_{0} \min \{ m(r), \mu(r) \}^{\frac{n-1}{n}} 
\end{equation}
for almost all $0<r<r_{0}$. Set $0 < t < r$ and define the competitor 
		\[
		F_{t} =
		\begin{cases}
			E \cup B_{t} \cap \Omega & \text{if } m(t) > \mu(t),\\ 
			E \setminus B_{t} \cap \Omega &  \text{otherwise.} 
		\end{cases}
		\]
		We note that in the first case $F_{t} \difsim E = B_{t} \cap \Omega \setminus E$, while in the second case $F_{t} \difsim E = B_{t} \cap \Omega \cap E$. In any case, we have $F_{t} \difsim E \subset \subset B_{r} \cap \Omega$. Thus, by the almost-minimality of $E$ in $\Omega$, and for almost all $0<t<r$, we infer that either
		\begin{align} \label{2a}
			\pr(E;B_{r}) & \leq \pr(F_{t};B_{r}) + \mu(t)^{\frac{n-1}{n}} \psi(r) \\ 
			\nonumber & \leq \pr(E;B_{r} \setminus \overline{B_{t}}) +  \Hau^{n-1}(\de B_{t} \cap \Omega \setminus E) + \mu(r)^{\frac{n-1}{n}} \psi(r) \, ,
		\end{align}
		or
		\begin{align} \label{2b}
			\pr(E;B_{r}) & \leq \pr(F_{t};B_{r}) + m(t)^{\frac{n-1}{n}} \psi(r) \\ 
			\nonumber & \leq \pr(E;B_{r} \setminus \overline{B_{t}}) + \Hau^{n-1}(\de B_{t} \cap \Omega \cap E) + m(r)^{\frac{n-1}{n}} \psi(r) \, .
		\end{align}
		where $\psi(r) := \psi_{\Omega}(E;0,r)$. Taking the limit as $t \nearrow r$ in \eqref{2a} and \eqref{2b}, and using \eqref{eq:relisop}, we deduce that, if $m(r) > \mu(r)$, then for almost all $0<r<r_{0}$ we have
		\begin{equation*}
			2\mu'(r) + \mu(r)^{\frac{n-1}{n}} \psi(r) \geq C_{0} \mu(r)^{\frac{n-1}{n}} \, ,
		\end{equation*}
		while otherwise, by \eqref{eq:relisop-bis}, we have
		\begin{equation*}
			2m'(r) + m(r)^{\frac{n-1}{n}} \psi(r) \geq C_{0} 
			m(r)^{\frac{n-1}{n}}\,.
		\end{equation*}
		Therefore, calling $s(r) := \min \{ m(r), \mu(r) \}$ and owing to the infinitesimality of $\psi(r)$ as $r\to 0$, we obtain
		\begin{equation*}
			\dfrac{s'(r)}{s(r)^{\frac{n-1}{n}}} \geq C \, ,
		\end{equation*}
		for $0 < r < r_{0}$ and for some explicit constant $C > 0$ depending on $C_{0}$. Integrating this inequality on the interval $(0,r)$ we obtain \eqref{eq:densityvol}. Then, the first inequality in \eqref{eq:densityper} follows from \eqref{eq:densityvol} and \eqref{eq:relisop}. Finally, the second inequality in \eqref{eq:densityper} follows from the observation that, taking the limit as $t \nearrow r$ in \eqref{2a} and possibly redefining $\overline{r}$ and $C$, we have
		\begin{align*}
			\pr(E;B_{r}) & \leq \Hau(\de B_{r} \cap \Omega \setminus E) + \mu(r)^{\frac{n-1}{n}} \psi(r) \leq C r^{n-1} \, ,
		\end{align*}
		for every $0 < r < r_{0}$.
\end{proof}

\section{The visibility property}
\label{sec:visibility}
In this section, we introduce the visibility property and its main consequences. In what follows, $\Omega \subset \R^n$ denotes an open set with Lipschitz boundary such that $0 \in \de \Omega$ and $\de \Omega$ is a graph in a neighborhood of $0$, as in \eqref{eq:omegagraph}. For notational convenience, we will only consider the visibility property at $0$, but of course, we could equally define the property at a generic point of $\de \Omega$. 
\begin{definition} \label{def:visib}
We say that $\Omega$ satisfies the \textit{visibility property} provided there exist $T>0$ and a function $u \in C^1([0,T))$\footnote{By $u\in C^{1}([0,T))$ we mean that $u\in C^{1}(0,T)$ and there exist finite limits of $u(t)$ and $u'(t)$ as $t\to 0$.} such that:
\begin{itemize}
\item[(V1)] $u(0)=u'(0)=0$ and $0\le u' \le 2^{-1}$;
\item[(V2)] The function
\[
\gamma_{u}(t) := t^{-1}\sup_{0<s\le t} \sqrt{\frac{u(s)}{s}+u'(s)}
\]
is summable on $(0,T)$;
\item[(V3)] for all $0 < t < T$, the segment joining the point $U_t = -u(t) \, e_{n}$ with a point $x$ belonging to $\de \Omega \cap B_{t}$ does not intersect $\Omega$.
\end{itemize}
\end{definition}
\begin{remark}\label{rem:V2strong}
We note that (V1) and (V2) imply that $u(t) = o(t)$ and 
\begin{equation}\label{eq:tgammazero}
t\gamma_{u}(t)\to 0\,,
\end{equation}
as $t\to 0$. Moreover, since by (V1) we have $0\le u(t) \le t/2$, we infer that
\begin{equation*}
0\le\frac{u(t)}{t^{2}} = t^{-1}\frac{u(t)}{t} \le t^{-1}\sqrt{\frac{u(t)}{t}} \le \gamma_{u}(t)\,.
\end{equation*}
Therefore the summability of $\gamma_u$ implies that of $t^{-2} u(t)$.
\end{remark}

In the following proposition, we rewrite the assumption (V3) in the form of a property involving the functions $\omega(x')$ and $u(t)$. This will be particularly useful when checking the visibility property for relevant classes of domains (see the examples at the end of the section).
\begin{proposition} \label{prop:equivisib}
The set $\Omega$ satisfies the property (V3) in Definition \ref{def:visib} if and only if, for any $\nu \in \de B'_1$ and for all $0 < t < T$, the slope $m_t(s)$ of the line connecting $U_t$ with $(s,\omega(s\nu))$, that is given by
\begin{equation*}
m_t(s) = \dfrac{\omega(s\nu) + u(t)}{s} \, ,
\end{equation*}
is non-increasing as a function of $s$, for $s > 0$ such that $s^{2} + \omega(s\nu)^{2} < t^2$.
\end{proposition}
\begin{proof}
Let us assume that (V3) holds, and set $\omega_{\nu}(s) = \omega(s\nu)$ for more simplicity. By contradiction, let $s_{1} < s_{2}$ be such that $s_{i}^{2} + \omega_{\nu}(s_i)^{2} < t^2$, for $i = 1,$ $2$, and $m_t(s_1) < m_t(s_2)$. By definition of $m_t$, we have
\begin{equation*}
\dfrac{\omega_{\nu}(s_1) + u(t)}{s_1} < \dfrac{\omega_{\nu}(s_2) + u(t)}{s_2} \, ,
\end{equation*}
and so equivalently
\begin{equation*}
\omega_{\nu}(s_1) < \dfrac{s_1}{s_2} \big(\omega_\nu(s_2) + u(t)\big) - u(t) \, .
\end{equation*}
This implies that the point
\begin{equation*}
x = \left(s_1 \nu , \dfrac{s_1}{s_2}\big(\omega_{\nu}(s_2) + u(t)\big) - u(t)\right)
\end{equation*}
is internal to $\Omega$ and lies on the segment connecting $(s_2 \nu , \omega_\nu(s_2))$. This contradicts (V3). 

Conversely, let us suppose that $m_t(s)$ is non-increasing in $s$, for $s > 0$ such that $s^2 + \omega_\nu(s)^2 < t^2$. Set $P(s) = (s \nu , \omega_\nu(s))$, then, arguing by contradiction, assume that $s_2 > 0$ is such that $s_2^2 + \omega_\nu(s_2)^2 < t^2$ and there exists $\lambda \in (0,1)$ with the property 
\begin{equation*}
(1 - \lambda) U_t + \lambda P(s_2) = (\lambda s_2 , \lambda (u(t) + \omega_\nu(s_{2})) - u(t)) \in \Omega \, .
\end{equation*}
Then $\omega_\nu(\lambda s_2) < \lambda (u(t) + \omega_\nu(s_2)) - u(t)$. By continuity, for all $\delta > 0$ there exists $\lambda_{\delta}\in (\lambda,1)$ such that
\begin{equation} \label{eq:lambdadelta}
\lambda_{\delta} (u(t) + \omega_\nu(s_2)) - u(t) - \delta < \omega_\nu(\lambda_{\delta} \, s_2) < \lambda_{\delta} (u(t) + \omega_\nu(s_2)) - u(t) \, .
\end{equation}
Since the segment $[U_t, P(s_2)]$ is compactly contained in $B_t$, by \eqref{eq:lambdadelta}, we can pick $\delta > 0$ small enough and a correspondent $\lambda_{\delta}$ such that
\begin{equation} \label{eq:condslambda}
P(\lambda_{\delta} s_2) \in B_t \, .
\end{equation}
Let $s_1 = \lambda_{\delta} s_2$. We observe that \eqref{eq:condslambda} and \eqref{eq:lambdadelta} imply
\begin{equation*}
m_t(s_1) < m_t(s_2) \, , \qquad s_1^2 + \omega_\nu(s_1)^2 < t^2 \, ,
\end{equation*}
and this contradicts our hypothesis. This completes the proof of the proposition.
\end{proof}
\begin{corollary} \label{cor:gradomega}
Assume that $\omega$ satisfies
\begin{equation} \label{eq:omegavis}
\langle x' , \nabla \omega (x') \rangle \leq \omega(x') + u(|x'|) \qquad \text{ for a.e. } x' \in B'_{\rho} \, ,
\end{equation}
where $u: (0,T) \rightarrow \R$ is a non-decreasing function satisfying properties (V1) and (V2). Then $\Omega$ satisfies the visibility property.
\end{corollary}
\begin{proof}
Since $\omega$ is Lipschitz, the function $m_{t}$ defined in the statement of Proposition \ref{prop:equivisib} is almost everywhere differentiable, thus $m_t$ is non-increasing if and only if $m_{t}'(s) \leq 0$ at almost every $s$. We observe that
\begin{equation*}
m_{t}'(s) = \dfrac{\omega_{\nu}'(s)}{s} - \dfrac{\omega_{\nu}(s) + u(t)}{s^{2}} \, ,
\end{equation*}
thus $m_{t}' \leq 0$ if and only if
\begin{equation} \label{eq:slope}
s \, \omega_{\nu}'(s) \leq \omega_{\nu}(s) + u(t) \, .
\end{equation}
The hypothesis \eqref{eq:omegavis} implies that 
\begin{equation} \label{eq:slopestrong}
s \, \omega_{\nu}'(s) \leq \omega_{\nu}(s) + u(s) \qquad \text{for almost all $0 < s < T \, .$}
\end{equation}
Hence if $s > 0$ is such that $s^2 + \omega_\nu(s)^2 < t^2$, then $s < t$, and by \eqref{eq:slopestrong}, since $u$ is non-decreasing, we obtain
\begin{equation*}
s \, \omega_{\nu}'(s) \leq \omega_{\nu}(s) + u(t) \, ,
\end{equation*}
which is precisely \eqref{eq:slope}. Consequently, (V3) is verified thanks to Proposition \ref{prop:equivisib}. 
\end{proof} 
	
\subsection{Existence of the tangent cone}
An important consequence of the visibility property is the existence of the tangent cone to $\Omega$ at $0$.
	\begin{proposition} \label{prop:tangcone}
		Let $\Omega \subset \R^{n}$ satisfy the visibility property. Then there exists the tangent cone to $\Omega$ at $0$, denoted by $\Omega_{0}$. More precisely, if we set $\Omega_{s} := s^{-1} \Omega$ for $s > 0$, we obtain
		\begin{equation} \label{eq:tangcone}
			\lim_{s \rightarrow 0^{+}} \Hdist(\Omega_{s} \cap B_{R}, \Omega_{0} \cap B_{R}) = 0 \, , \qquad \text{for all $R > 0$.}
		\end{equation}
	\end{proposition}
	\begin{proof}
		Let us fix $\nu \in \de B'_{1}$, and let
		\begin{equation*}
			\omega_{\nu}(s) = \omega(s\nu) \, , \qquad s \geq 0 \, ,
		\end{equation*}
		where $\omega$ is the function realizing \eqref{eq:omegagraph}. Let $s > 0$ be a point where $\omega_{\nu}$ is differentiable and let $t = \sqrt{s^{2} + \omega_{\nu}(s)^{2}}$. By (V3) in Definition \ref{def:visib}, we deduce that the slope of the line connecting $(s,\omega_{\nu}(s))$ with $U_t$ needs to be bounded below by $\omega_{\nu}'(s)$, that is,
\[
\omega_{\nu}'(s) \leq \dfrac{\omega_{\nu}(s) + u(t)}{s} \, .
\]
We set $L= \sqrt{1 + \Lip(\omega)^{2}}$ and observe that $t \leq L s$. Since $u$ is non-decreasing by (V1), we infer 
\begin{equation*}
\omega_{\nu}'(s) \leq \dfrac{\omega_{\nu}(s) + u(L s)}{s}\,, 
\end{equation*}
hence we get
\begin{equation} \label{eq:monotomega}
\left( \dfrac{\omega_{\nu}(s)}{s} \right)' = \dfrac{\omega_{\nu}'(s)}{s} - \dfrac{\omega_{\nu}(s)}{s^{2}} \leq \dfrac{u(Ls)}{s^{2}} \leq L^{2} \dfrac{u(Ls)}{(L s)^{2}} \, .
\end{equation}
We integrate \eqref{eq:monotomega} between $s_{1} < s_{2}$ thanks to (V2) (see Remark \ref{rem:V2strong}), and obtain
\begin{equation*}
\dfrac{\omega_{\nu}(s_{2})}{s_{2}} - \dfrac{\omega_{\nu}(s_{1})}{s_{1}} \leq L^{2} \int_{s_{1}}^{s_{2}} \dfrac{u(Ls)}{(Ls)^{2}} \, ds = L \int_{L s_{1}}^{L s_{2}} \dfrac{u(t)}{t^{2}} \, dt \, .
\end{equation*}
Thus we conclude that the function
\begin{equation*}
s\mapsto \dfrac{\omega_{\nu}(s)}{s} - L \int_{0}^{L s} \dfrac{u(t)}{t^{2}} \, dt
\end{equation*}
is monotonically non-increasing in $s$ and bounded by $\Lip(\omega) + L \int_{0}^{L s} t^{-2} u(t) \, dt$, for $0 < s < L^{-1} T$. Therefore there exists 
\begin{equation*}
D_{\nu}^{+} \omega(0) := \lim_{s \rightarrow 0^{+}} \dfrac{\omega_{\nu}(s)}{s} = \lim_{s \rightarrow 0^{+}} \dfrac{\omega_{\nu}(s)}{s} - L \int_{0}^{L s} \dfrac{u(t)}{t^{2}} \, dt \in \R\, .
\end{equation*}
Let us define
\begin{equation*}
\omega_{0}(x') = 
\begin{cases}
|x'| D_{\frac{x'}{|x'|}}^{+} \omega (0) & \text{if } x' \neq 0,\\ 
0 &  \text{if } x' = 0 \, . 
\end{cases}
\end{equation*}
		The function $\omega_{0}$ is $1$-homogeneous, therefore the set
		\begin{equation*}
			\Omega_{0} = \{ x \in \R^{n} \, : \, x_{n} > \omega_{0}(x') \}
		\end{equation*}
		is a cone with vertex at $0$. Now, for all $s > 0$, we set
		\begin{equation*} 
			\omega_{s}(x') = 
			\begin{cases}
				\dfrac{\omega(sx')}{s} & \text{if } x' \neq 0,\\ 
				0 &  \text{if } x' = 0 \, . 
			\end{cases}
		\end{equation*}
		It is immediate to observe that $\omega_{s}(0) = 0$ and, setting $t = s |x'|$,
		\begin{align*}
			\omega_{s}(x') = \dfrac{\omega \left(t (x' / |x'|)  \right)}{t} \, |x'| \rightarrow |x'| \, D^{+}_{\frac{x'}{|x'|}} \omega (0) = \omega_{0}(x') \, , \qquad \text{as $s \ra 0^+ \, .$}
		\end{align*}
		Since $\{ \omega_{s} \}_{s > 0}$ is a one-parameter family of locally equi-bounded and equi-Lipschitz functions that pointwisely converge to $\omega_{0}$ as $s \rightarrow 0^{+}$, we can apply Lemma \ref{lem:uniflem} to conclude that this convergence is locally uniform. This easily implies the Hausdorff convergence stated in \eqref{eq:tangcone}.
	\end{proof}
\begin{remark}
We note that for the proof of Proposition \ref{prop:tangcone}, the hypothesis (V2) of Definition \ref{def:visib} can be replaced by the weaker hypothesis of summability of $t^{-2}u(t)$ on $(0,T)$. We also observe that, if $\Omega$ is convex, then the existence of the tangent cone $\Omega_0$ is always granted, even though (V2) may not be satisfied.
\end{remark}
	
\subsection{An off-centric visibility property}
The next lemma shows that the assumption (V3) in Definition \ref{def:visib} can be replaced by an equivalent assumption, where off-centric balls are taken instead of balls centered at $0$. This off-centric visibility property will be useful later on. 
\begin{lemma} \label{lem:visibalt}
The following properties are equivalent:
\begin{itemize}
\item[(i)] $\Omega$ satisfies the visibility property;
\item[(ii)] There exist $R > 0$ and a function $v \in C^{1}(0,R)$ satisfying properties (V1), (V2) of Definition \ref{def:visib}, and
\begin{center}
\begin{minipage}{.8\textwidth}
(V3') for all $0 < r < R$, any segment joining the point $V_r = -v(r) \, e_{n}$ with a point $x$ belonging to $\de \Omega \cap B_{r}(V_r)$ does not intersect $\Omega$.
\end{minipage}
\end{center}
\end{itemize}
\end{lemma}
\begin{proof}
We prove that (i) implies (ii). Let $z(t) = t - u(t)$ where $u$ is as in Definition \ref{def:visib}. We can find $0 < T' < T$ such that $z(t)$ is an increasing $C^1$ diffeomorphism of the interval $(0,T')$ with the property
\begin{equation*}
\dfrac{1}{2} \, t \leq z(t) \leq t \, .
\end{equation*}
Let $R = z(T')$. Then we can consider the inverse $z^{-1}$ of $z$ in $(0,R)$, which is an increasing diffeomorphism such that
\begin{equation} \label{eq:estz-1}
r \leq z^{-1}(r) \leq 2 \, r \, .
\end{equation}
Setting $U_t = - u(t) e_n$, it follows that $B_{t - u(t)}(U_t) \subset B_{t}$, for all $0 < t < T$, thus (V3) holds for all points $x \in \de \Omega \cap B_{t - u(t)}(U_t)$. We then have that
\begin{equation*}
B_{r}(U_{z^{-1}(r)}) \subset B_{z^{-1}(r)} \, , \qquad \text{for any $0 < r < R$.}
\end{equation*}
Any line segment joining $x \in \de \Omega \cap B_{z^{-1}(r)}$ with $U_{z^{-1}(r)}$ does not intersect $\Omega$, hence the same property holds for any $x \in \de \Omega \cap B_{r}(U_{z^{-1}(r)})$. Let
\begin{equation*}
v(r) = u(z^{-1}(r)) \, , \qquad 0 < r < R \, .
\end{equation*}
It is clear that (V3') holds. By \eqref{eq:estz-1}, up to possibly reducing the value of $R$, (V2) and $v' \leq 2^{-1}$ follow. Since both $u$ and $z^{-1}$ are non-decreasing, $v$ is non-decreasing, thus also (V1) is satisfied. A completely analogous argument shows that (ii) implies (i), and the proof is concluded.
\end{proof}
	For all $0 < r < R$, we define
	\begin{equation} \label{eq:Cr}
		\C_{r} = \{ V_r + t(z - V_r) \, : \, z \in \de B_{r}(V_r) \cap \Omega \, , \, t > 0 \} \, .
	\end{equation}
	The set $\C_{r}$ is an open cone with vertex at $V_r$.
	\begin{lemma} \label{lem:Crcontains}
		Assume that $\Omega$ satisfies the (off-centric) visibility property. Then, for all $0 < r < R$, the cone $\C_{r}$ contains $\Omega \cap B_{r}(V_r)$.
	\end{lemma}
	\begin{proof}
		By contradiction, let $x \in (\Omega \cap B_{r}(V_r)) \setminus \C_{r}$. For all $0 < t \leq r$, let
		\begin{equation*}
			x_{t} = V_r + t \, \dfrac{x - V_r}{|x - V_r|} \, ,
		\end{equation*}
		and note that $x = x_{t_{0}}$ for a suitable $0 < t_{0} < r$. It's clear that $x_{r} \notin \Omega$, otherwise $x_{t_{0}}$ would belong to $\C_{r}$, for all $t > 0$, which contradicts our assumption. We can then select a value $s \in (t_{0}, r]$ such that $x_{s} \in \de \Omega$. This leads to a contradiction with the visibility because the segment joining $V_r$ and $x_{s} \in B_{r}(V_r) \cap \de \Omega$ contains $x_{t_{0}} = x \in \Omega$.
	\end{proof}
	
\subsection{Foliation by off-centric spheres}
Let us consider the family of off-centric balls $B_r(V_r)$, with $V_r = - v(r) e_n$, for $0<r<R$. 
By the Implicit Function Theorem we can easily show the existence of a smooth function $\phi$, such that $\de B_{r}(V_r)$ is the $r$-level set of $\phi$. This means that the punctured ball $B_R(V_R) \setminus \{ 0 \}$ is foliated by the spheres $\de B_{s}(V_s) = \phi^{-1}(s)$ for $0 < s < R$.
\begin{lemma} \label{lem:coarea}
There exists a function $\phi \in C^{1}(B_{R}(V_R) \setminus \{ 0 \})$ such that $0 \leq \phi < R$ and
\begin{equation*}
\de B_{r}(V_r) = \phi^{-1}(r) \, , \qquad \text{for any $0 < r < R$.} 
\end{equation*}
In particular, for any $x \in B_{R}(V_R) \setminus \{ 0 \}$ we have
\begin{equation} \label{eq:gradphi}
\dfrac{\nabla \phi(x)}{|\nabla \phi(x)|} = \dfrac{x - V_{\phi(x)}}{|x - V_{\phi(x)}|}
\end{equation}
and 
\begin{align}
\left| \nabla \phi (x) - \dfrac{x}{|x|} \right| \leq  4 \, \sqrt{\dfrac{v(\phi(x))}{\phi(x)} + v'(\phi(x))} \, . \label{eq:Dphi-1}
\end{align}
\end{lemma}
\begin{proof}
If $v(r)$ is identically $0$, then there is nothing to prove because $\phi(x) = |x|$ in this case. We then suppose $v \neq 0$. Let us start by proving the existence of the function $\phi$. We observe that, for any $x \in B_{R}(V_R)$, there exists a unique $r = r_{x} \in [0,R)$ such that $x \in \de B_{r}(V_r)$. Indeed, if $x = 0$, we can take $r = 0$. Otherwise, let $F: (B_{R}(V_R) \setminus \{ 0 \}) \times (0,R) \rightarrow \R$ be the function defined by
\begin{equation} \label{eq:defF}
F(x,r) = |x - V_r|^{2} - r^{2} \, .
\end{equation}
It is immediate to observe that $F$ is continuous. Moreover,
\begin{equation*}
F(x,0) = |x|^{2} > 0 \quad\text{and}\quad F(x,R) < 0 \, , \qquad \text{for all } x \in B_{R}(V_R) \setminus \{ 0 \} \, .
\end{equation*}
Hence we can find $r \in (0,R)$ such that $F(x,r) = 0$, i.e. such that $x \in \de B_{r}(V_r)$. Let us show the uniqueness. Indeed, if $r$, $r' \in (0,R)$ have the property that
\begin{equation*}
x \in \de B_{r}(V_r) \cap \de B_{r'}(V_{r'}) \, ,
\end{equation*}
then we get
\begin{equation*}
|r - r'| = ||x - V_r| - |x - V_{r'}|| \leq |V_r - V_{r'}| = |v(r) - v(r')| \leq \dfrac{1}{2} \, |r - r'| \, ,
\end{equation*}
thus we must have $r = r'$. Now we can define $\phi(x) = r_{x}$. Let us show that $\phi \in \cC^{1}(B_{R}(R) \setminus \{ 0 \})$. To do so, we note that $\phi$ is implicitly defined by
\begin{equation*} 
F(x, \phi(x)) = 0 \, ,
\end{equation*}
where $F$ is the function defined in \eqref{eq:defF}. Easy computations give
\begin{equation*}
\de_{r} F(x,r) = -2r + 2 v'(r)(x_{n} + v(r)) \, .
\end{equation*}
Therefore, if we assume $F(x,r) = 0$ (that is, $x \in B_{r}(V_r)$) we obtain
\begin{equation*}
\de_{r} F (x,r) = -2r + 2 v'(r)(x_{n} + v(r)) \leq 2r + |x_{n} + v(r)| \leq -r \, ,
\end{equation*}
where the first inequality follows from the assumption $0 \leq v' \leq 2^{-1}$, while the second inequality from
\begin{equation*}
|x_{n} + v(r)| = |\langle x - V_r , e_{n} \rangle| \leq |x - V_r| = r \, .
\end{equation*}
By the Implicit Function Theorem we deduce that $\phi \in \cC^{1}(B_{R}(V_R) \setminus \{ 0 \})$. The identity \eqref{eq:gradphi} is a consequence of the fact that, if $\phi(x) = r$, then the vector $\nabla \phi(x)$ is orthogonal to the level set $\de B_{r}(V_r) = \phi^{-1}(r)$ at $x$. 

Let us now prove \eqref{eq:Dphi-1}. We first observe that, if $\phi(x) = r > 0$, then
\begin{equation} \label{eq:nablaphi}
\nabla \phi (x) = - \dfrac{\de_{x} F (x,r)}{\de_{r} F (x,r)} = \dfrac{x + v(r) e_{n}}{r - v'(r)(x_{n} + v(r))}
\end{equation}
Then \eqref{eq:nablaphi} yields
\begin{align} \label{eq:est1}
\left| \nabla \phi(x) - \dfrac{x}{|x|} \right|^{2} & = \dfrac{|x + v e_{n}|^{2}}{(r - v'(x_{n} + v))^{2}} + 1 - 2 \, \left \langle \dfrac{x + v e_{n}}{r - v'(x_{n} + v)} , \dfrac{x}{|x|} \right \rangle \nonumber \\
& = \dfrac{|x + v e_{n}|^{2}}{(r - v'(x_{n} + v))^{2}} + 1 - 2 \, \dfrac{|x + v e_{n}|^{2} - v(x_{n} + v)}{(r - v'(x_{n} + v))|x|} \nonumber \\
& = \dfrac{|x| \, r^{2} + |x|(r - v'(x_{n} + v))^{2} - 2(r - v'(x_{n} + v))[r^{2} - v(x_{n} + v)]}{|x|(r - v'(x_{n} + v))^{2}} \nonumber\\
& = \dfrac{\frac{|x|}{r} + \frac{|x|}{r} \left( 1 - \frac{v'(x_{n} + v)}{r} \right)^{2} - 2 \, \left( 1 - \frac{v'(x_{n} + v)}{r} \right) \left( 1 - \frac{v}{r} \, \frac{x_{n} + v}{r} \right)}{\frac{|x|}{r} \left( 1 - \frac{v'(x_{n} + v)}{r} \right)^{2} } \nonumber \\
& = \dfrac{1 +  \left( 1 - \frac{v'(x_{n} + v)}{r} \right)^{2} - 2 \, \dfrac{r}{|x|} \left( 1 - \frac{v'(x_{n} + v)}{r} \right) \left( 1 - \frac{v}{r} \, \frac{x_{n} + v}{r} \right)}{\left( 1 - \frac{v'(x_{n} + v)}{r} \right)^{2} } \, , 
\end{align}
where the short forms $v=v(r)$ and $v'=v'(r)$ have been used.
Next we observe that $\big||x| - r\big| = \big|x - |x - V_r|\big|$, hence
\begin{equation}\label{eq:est2}
r-v(r)\le |x| \le r+ v(r)\,.
\end{equation}
Exploiting \eqref{eq:est2} in \eqref{eq:est1} and recalling that $v'(r) \le 1/2$ and $|x_{n}+v(r)|\le r$, we get
\begin{align*}
\left| \nabla \phi(x) - \dfrac{x}{|x|} \right|^{2} &\leq \dfrac{ 2 - 2 \, \left( 1 - \frac{v(r)}{r+v(r)} \right) ( 1 - v'(r)) \left( 1 - \frac{v(r)}{r}\right)}{(1 - v'(r))^{2} }\\
&\le 8\left(1-\left(1-\frac{v(r)}{r}\right)^{2}(1-v'(r))\right)\\
&\le 16\left(v'(r) + \frac{v(r)}{r}\right)\,,
\end{align*}
and this concludes the proof.
\end{proof}

\subsection{Some examples} \label{sec:examples}
We note that there are examples of sets with Lipschitz boundary that do not satisfy the visibility property at $0$, like for instance the epigraph of the function 
\[
\omega(x) = \begin{cases}
x^{2}\sin(|x|^{-1}) & \text{if }x\neq 0\\
0 & \text{otherwise.}
\end{cases}
\]
Indeed, one can easily check that at $x_{k} = \frac{1}{(2k+1)\pi}$ for $k\in \N$ we have $\omega(x_{k}) = 0$ and $\omega'(x_{k})=1$, hence any visibility function $u$ for which (V3) holds must satisfy $u(x_{k})\ge x_{k}$, which contradicts both  (V1) and (V2) (see Remark \ref{rem:V2strong}).

In the following, we exhibit some examples of domains for which the visibility holds. We recall that $\omega_{\nu}(s) = \omega(s \nu)$ for $s \geq 0$.

\begin{example}[Lipschitz cones and outer star-shaped sets] \label{ex:Lipcones}
Let $\Omega$ be either a cone with respect to $0$, or such that its complement $\R^{n}\setminus \Omega$ is locally star-shaped with respect to $0$. It is immediate to check that $\Omega$ satisfies the visibility property with visibility function $u(t) \equiv 0$.
\end{example}
	
\begin{example}[$C^{1,\beta}$-sets]
Let $\Omega$ have $C^{1,\beta}$ boundary and assume $0\in \de\Omega$. We show that $\Omega$ satisfies the visibility property. Up to a rotation we can assume that $\Omega$ admits a representation as in \eqref{eq:omegagraph} with $\omega \in C^{1,\beta}(B'_{\rho})$. By Corollary \ref{cor:gradomega}, it is enough to show that $\omega$ satisfies \eqref{eq:omegavis}. Since $\nabla \omega$ is $\beta$-H\"older, we have
\begin{equation} \label{eq:Holder1}
\langle x' , \nabla \omega(x') \rangle \leq \langle x' , \nabla \omega (0) \rangle + u(|x'|) \, .
\end{equation}
where $u(t) = C \, t^{1 + \beta}$ for some $C > 0$. Set $\overline{\omega}(x') := \omega(x') - \langle \nabla \omega (0) , x' \rangle$ and note that $\overline{\omega}$ is $C^{1,\beta}$, $\overline{\omega}(0) = 0$, $\nabla \overline{\omega} (0) = 0$, and
		\begin{align} \label{eq:Holder2}
			|\omega(x') - \langle x' , \nabla \omega (0) \rangle| & \leq \max_{|y'| \leq |x'|} |\nabla \overline{\omega} (y')||x'| \nonumber \\
			& \leq (|\nabla \overline{\omega} (0)| + C |x|^{\beta})|x'| \\
			& = C \, |x'|^{1 + \beta} \, . \nonumber
		\end{align}
		Putting together \eqref{eq:Holder1} and \eqref{eq:Holder2}, we finally get
		\begin{equation*}
			\langle x' , \nabla \omega (x') \rangle \leq \omega(x') + u(|x'|), \, \qquad \text{ for a.e. } x' \in B'_{\rho} \, ,
		\end{equation*}
		which is precisely \eqref{eq:omegavis}. Since trivially $u$ is non-decreasing and satisfies (V2), by Corollary \ref{cor:gradomega} we infer that $\Omega$ satisfies the visibility property.
\end{example}

\begin{example}[Convex sets satisfying (V2)] \label{ex:convex}
Let $\Omega$ be a convex set with $0 \in \de \Omega$. For $s > 0$, let $\Omega_{s} := s^{-1} \Omega$ and $\Omega_{0} := \bigcup_{s > 0} \Omega_{s}$. The set $\Omega_{0}$ is the tangent cone to $\Omega$ at $0$. Let
\begin{equation} \label{eq:Hdistcontrol}
u(r) := \Hdist(\Omega \cap B_{r}, \Omega_{0} \cap B_{r}) \, ,
\end{equation}
and assume that $u(r)$ satisfies (V2). We observe that $u$ is non-decreasing in $r$. Let us prove that $\Omega$ satisfies the visibility property. As before, we assume that $\Omega$ admits a graphical representation as in \eqref{eq:omegagraph} with the further property that $\omega: B'_{\rho} \rightarrow \R$ is convex. Using the notations introduced in the proof of Proposition \ref{prop:tangcone}, by the convexity of $\omega$, we can define 
		\begin{equation*}
			\omega_{0}(x') = 
			\begin{cases}
				|x'| D_{\frac{x'}{|x'|}}^{+} \omega (0) & \text{if } x' \neq 0,\\ 
				0 &  \text{if } x' = 0 \, ,
			\end{cases}
		\end{equation*}
		and deduce that
		\begin{equation*}
			\Omega_{0} = \{ x \in \R^{n} \, : \, x_{n} > \omega_{0}(x') \} \, .
		\end{equation*}
		By the definition of $u$ given in \eqref{eq:Hdistcontrol}, we have
		\begin{equation} \label{eq:Hdistcontrolomega}
			||\omega - \omega_0||_{L^{\infty}(B'_r)} \leq C u(r) \, , \qquad \text{for some $C > 0 \, .$} 
		\end{equation}
		Owing to Corollary \ref{cor:gradomega}, the visibility property can be proved by showing that \eqref{eq:omegavis} holds. Thanks to the convexity of $\omega$, for all $\nu \in \de B'_{1}$, we have
		\begin{equation} \label{eq:estdir1}
			D^{+}_{\nu} \omega (0) \leq \omega'_\nu \left( \frac{s}{2}^+ \right) := \lim_{\sigma \rightarrow 0^+} \dfrac{\omega_\nu(s/2 + \sigma) - \omega_\nu(s/2)}{\sigma} \, , \qquad \text{for all $0 \leq s < \rho \, .$}
		\end{equation}
		Moreover, for all $0 < \sigma < s/2$, by \eqref{eq:Hdistcontrolomega} and the convexity of $\omega_\nu$, we have
		\begin{align} \label{eq:estdir2}
			\dfrac{\omega_\nu(s/2 + \sigma) - \omega_\nu(s/2)}{\sigma} & \leq \dfrac{\omega_\nu (s) - \omega_\nu(s/2)}{s/2} \nonumber \\
			& = 2 \, \dfrac{\omega_\nu(s)}{s} - \dfrac{\omega_\nu(s/2)}{s/2} \nonumber\\
			& = 2 \, \dfrac{\omega_\nu(s)}{s} - 2 D^{+}_{\nu} \omega (0) + D^{+}_{\nu} \omega (0) - \dfrac{\omega_\nu(s/2)}{s/2} + D^{+}_{\nu} \omega (0) \\
			& = \dfrac{2}{s} \, \left(\omega_{\nu}(s) - \omega_0(s \nu) + \omega_0(s \nu / 2) - \omega_{\nu}(s/2) \right) + D^{+}_{\nu} \omega (0) \nonumber \\
			& \leq D^{+}_{\nu} \omega (0) + \tilde{C} \, \dfrac{u(s)}{s} \, , \qquad \text{for some $\tilde{C} > 0 \, .$} \nonumber
		\end{align}
		Putting together \eqref{eq:estdir1} and \eqref{eq:estdir2}, we obtain
		\begin{equation*}
			\left| D^{+}_{\nu} \omega (0) - \omega'_\nu \left( \frac{s}{2}^+ \right) \right| \leq \tilde{C} \, \dfrac{u(s)}{s} \, . 
		\end{equation*}
		This suffices for us to conclude. In fact, if $x' \in B'_{\rho} \setminus \{ 0 \}$ is such that $\omega$ is differentiable at $x'$, setting $s := |x'|$, $\nu := x'/|x'|$, for some $C > 0$, we achieve
		\begin{align*}
			\langle x' , \nabla \omega (x') \rangle & = \langle s \nu , \nabla \omega(s \nu) \rangle = s \, \omega'_\nu (s^+) \\
			& \leq s \, D^{+}_{\nu} \omega (0) + C \, u(s) \\
			& \leq \omega(s \nu) + C \, u(s) = \omega(x') + C \, u(|x'|) \, .
		\end{align*}
		Since $\omega$ is convex, it is differentiable almost everywhere in $B'_{\rho}$, and so \eqref{eq:omegavis} is verified. Moreover, $u$ is non-decreasing and satisfies (V2) by our assumption, thus we conclude.
\end{example}

\begin{example}
	For $i \in \N$, let $y_i = 2^{-i}$ and $\tilde{y}_i = 2^{-i} + 4^{-i}$.
    We observe that $y_{i+1} < \tilde{y}_{i+1} < y_i$. We define $P_i = (y_i, \tilde{y}_i)$, $Q_i = (\tilde{y}_i, \tilde{y}_i)$, and consider the polygonal curve formed by the segments $\overline{P_i Q_{i+1}}$, $\overline{Q_{i+1} P_{i+1}}$, for all $i \geq 0$. It is immediate to observe that this curve coincides with the graph of the function $\omega: [0,1] \ra \R$ defined by $\omega(0)=0$ and 
    \begin{equation} \label{eq:defomega}
    	\omega(y) =
    	\begin{cases}
    		2^{-(i+1)} + 4^{-(i + 1)}  & \text{if } y \in (y_{i+1}, \tilde{y}_{i+1}],\\ 
    		a_i \, y - b_i &  \text{if } y \in (\tilde{y}_{i+1}, y_i] \, ,
    	\end{cases}
    \end{equation}
where
    \begin{equation*}
    	a_i = \dfrac{1 + 3 \cdot 2^{-(i+1)}}{1 - 2^{-(i + 1)}} \, , \qquad b_i = \dfrac{4^{-i} + 2^{-(3i + 1)}}{1 - 2^{-(i+1)}} \, .
    \end{equation*}
    \begin{figure}[ht!]
    \centering
    \includegraphics{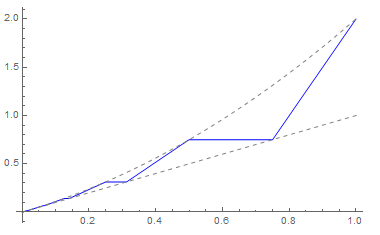}
    \, \, \, \, \, \, \caption{The graph of $\omega(y)$ "bounces" between the graphs of $y$ and $y + y^2$.}\label{tab:graphomega}
    \end{figure}
    Let $\Omega \subset \R^2$ be an open set such that
    \begin{equation*}
    	\Omega \cap ((-1,1) \times \R) = \{ x = (y,z) \in \R^2 \, : \, z > \omega(|y|) \} \, .
    \end{equation*}
    Let us show that $\Omega$ satisfies the visibility condition at $0$. Owing to Corollary \ref{cor:gradomega}, it suffices to show the existence of a non-decreasing function $u : (0,1) \ra \R$ satisfying (V1), (V2) and such that
    \begin{equation} \label{eq:omegavisex}
    	y \, \omega'(y) \leq \omega(y) + u(y) \, , \qquad \text{ for a.e. } y \in (0,1) \, .
    \end{equation}
By \eqref{eq:defomega}, for every $i \in \N$, we have
\begin{equation*}
	y \, \omega'(y) =
	\begin{cases}
		0 & \text{if } y \in (y_{i+1}, \tilde{y}_{i+1}],\\ 
		a_i \, y &  \text{if } y \in (\tilde{y}_{i+1}, y_i] \, .
	\end{cases}
\end{equation*}
Thus, in order to verify \eqref{eq:omegavisex}, it suffices to choose a function $u = u(y)$ greater or equal than the function
\begin{equation*}
	\overline{u}(y) =
	\begin{cases}
		0 & \text{if } y \in (y_{i+1}, \tilde{y}_{i+1}],\\ 
	    b_i &  \text{if } y \in (\tilde{y}_{i+1}, y_i] \, .
	\end{cases}
\end{equation*}
We look for $\alpha > 0$ such that
\begin{equation} \label{eq:ualphageq}
	u_\alpha(y) = \alpha y^2 \geq \overline{u}(y) \, .
\end{equation}
To obtain \eqref{eq:ualphageq}, it suffices to impose that
\begin{equation*}
	u_\alpha(\tilde{y}_{i+1}) = \alpha \, \tilde{y}_{i+1}^2 \geq b_i \, .
\end{equation*}
It is immediate to observe that
\begin{equation*}
	\alpha \, \tilde{y}_{n+1}^2 \geq \dfrac{\alpha}{4} \, 2^{-2i} \, , \qquad b_i \leq 4 \cdot 2^{-2i} \, ,
\end{equation*}
hence, if for instance we take $\alpha = 16$, \eqref{eq:ualphageq} holds. Since the function $u(y) = 16 \, y^2$ fulfills (V1) and (V2), we conclude that $\Omega$ satisfies the visibility condition at $0$.
\end{example}

\section{Free-boundary monotonicity}
\label{sec:BMF}
The present section is devoted to showing a free-boundary monotonicity property for local almost-minimizers of $\pr$ at a point $x_0\in\de\Omega$ satisfying the visibility property up to an isometry (hence, from now on, we will directly assume $x_{0}=0$). 
Following the notation of the previous section, given $0 < r_{1} < r_{2} < R$, we recall that $V_{r}=-v(r)e_{n}$ and define
\begin{equation*}
\cA_{r_{1},r_{2}} := B_{r_{2}}(V_{r_{2}}) \setminus \overline{B_{r_{1}}(V_{r_{1}})} = \phi^{-1}(r_{1},r_{2})\,,
\end{equation*}
where $\phi(x)$ is the function defined in Lemma \ref{lem:coarea}. We also conveniently introduce some further notation. Given $f \in BV_{\loc}(\Omega)$ and $0<r_{1}<r_{2}$, we set
\[
\mu_{f}(r) = \dfrac{|Df|(\Omega \cap B_{r}(V_{r}))}{r^{n-1}}
\]
and
\begin{align} \label{eq:errtermfBV}
G(f;r_1, r_2) &= \int_{r_1}^{r_2} \dfrac{n-1}{\rho^{n}} \, \int_{\Omega \cap B_{\rho}(V_{\rho})} (|\nabla \phi| - 1) \, d |Df| \, d \rho \\ 
& \qquad + \dfrac{1}{r_{2}^{n-1}} \, \int_{\Omega \cap B_{r_{2}}(V_{r_2})} (|\nabla \phi| - 1) \, d|Df| - \dfrac{1}{r_{1}^{n-1}} \, \int_{\Omega \cap B_{r_1}(V_{r_1})} (|\nabla \phi| - 1) \, d|Df|  \, . \nonumber
\end{align}
In the next proposition, we combine the visibility property with an upper bound on $\mu_{f}(r)$ and obtain the finiteness of $\lim_{\rho\to 0} G(f;\rho,r)$.
\begin{proposition}\label{prop:Gbound1}
Let $\Omega \subset \R^{n}$ satisfy the visibility property as in Definition \ref{def:visib}, and let $f \in BV_{\loc}(\Omega)$. Assume that $\mu_{f}(r)\le C$ for some constant $C>0$ and for all $r\in (0,R)$. Then for $r\in (0,R)$ the limit
\begin{align*}
G(f;r) &:= \lim_{\rho\to 0} G(f;\rho,r)\\\nonumber
&= \int_{0}^{r} \dfrac{n-1}{\rho^{n}} \, \int_{\Omega \cap B_{\rho}(V_{\rho})} (|\nabla \phi| - 1) \, d |Df| \, d \rho  + \dfrac{1}{r^{n-1}} \, \int_{\Omega \cap B_{r}(V_{r})} (|\nabla \phi| - 1) \, d|Df|  
\end{align*}
exists and is finite.
\end{proposition}
\begin{proof}
By \eqref{eq:Dphi-1}, for all $x\in B_{\rho}(V_{\rho})$ we have
\begin{align*}
\left||\nabla \phi(x)| - 1\right| &\le \left|\nabla \phi(x) - \frac{x}{|x|}\right|\le 4\sqrt{\frac{v(\phi(x))}{\phi(x)} + v'(\phi(x))} \le 4\sup_{0<r\le \rho} \sqrt{\frac{v(r)}{r} + v'(r)} = 4\rho \gamma_{v}(\rho)\,,
\end{align*}
where $\gamma_{v}$ is the function defined in the visibility property (V2). Then, using the upper bound on $\mu_{f}$, for $0<\rho<r<R$ we obtain
\begin{align*}
\left|\rho^{1-n} \int_{\Omega\cap B_{\rho}(V_{\rho})} (|\nabla \phi|-1)\, d|Df|\right| \le 4C\rho \gamma_{v}(\rho)\,.
\end{align*}
Thanks to the summability of $\gamma_{v}$ (see property (V2) of the visibility property) and to \eqref{eq:tgammazero}, from the last inequality we easily get the proof of the proposition. 
\end{proof}
Our monotonicity formula will then follow from the general inequality proved in the next theorem.
\begin{theorem}[Monotonicity inequality] \label{thm:MonotonicityInequality}
Let $\Omega \subset \R^{n}$ satisfy the visibility property, and let $f \in BV_{\loc}(\Omega)$. Then, for $R > 0$ small enough, for almost every $0 < r_{1} < r_{2} < R$, we have
\begin{align} \label{eq:finalmonotfBV}
& \left( \int_{\Omega \cap \cA_{r_{1},r_{2}}} \phi^{1-n} \left|\left\langle \nu_f, \nabla \phi \right\rangle \right| \, d |Df| \right)^{2} \leq 2\left( \int_{\Omega \cap \cA_{r_{1},r_{2}}} \frac{|\nabla \phi(x)|}{\phi(x)^{n-1}}\, d |Df| \right)\nonumber \\
& \qquad \qquad \qquad \cdot \bigg[\mu_{f}(r_{2}) - \mu_{f}(r_{1}) + \int_{r_{1}}^{r_{2}} \dfrac{n-1}{\rho^{n}} \, \Psi_{\Omega}(f;B_{\rho}(V_{\rho})) d \rho + G(f;r_1,r_{2}) \bigg] \, ,
\end{align}
where $\nu_f$ is such that $Df = \nu_{f}|Df|$.
\end{theorem}
\begin{proof}
We start by assuming $f \in BV(\Omega) \cap C^{1}(\Omega)$. For all $0 < r < R$ and $x \in \C_{r} \cap B_{r}(V_{r})$, where $\C_{r}$ is defined in \eqref{eq:Cr}, we let
\begin{equation*}
Y_{r}(x) = V_r + r \, \dfrac{x - V_r}{|x - V_r|} \, .
\end{equation*}
Standard computations yield
\begin{align} \label{eq:DYr}
DY_{r}(x) & = r \, \left[ \dfrac{1}{|x - V_r|} D(x - V_r) + (x - V_r) \otimes \nabla |x - V_r|^{-1} \right] \\
& = \dfrac{r}{|x - V_r|} \, \left[ \Id - \dfrac{x - V_r}{|x - V_r|} \otimes \dfrac{x - V_r}{|x - V_r|} \right] \, . \nonumber
\end{align}
We define
\begin{equation*}
g_{r}(x) = f(Y_{r}(x))\qquad \text{for all $x \in \C_{r} \, ,$}
\end{equation*}
then the ``off-centric conical competitor'' is
\begin{equation*} 
f_{r}(x) = 
\begin{cases}
g_{r}(x) & \text{if $x \in \C_{r} \cap B_{r}(V_r)$} \\
f(x) & \text{if $x \in \Omega \setminus B_{r}(V_r) \, .$}
\end{cases}
\end{equation*}
		By definition, $f_{r}$ coincides with $f$ in $\Omega \setminus B_{r}(V_r)$, hence we infer
		\begin{equation} \label{eq:mingap}
			\Psi_{\Omega}(f;B_{r}(V_r)) \geq |Df|(\Omega \cap B_{r}(V_r)) - |Df_{r}|(\Omega \cap B_{r}(V_r)) \, .
		\end{equation}
		Then, by \eqref{eq:mingap}, we deduce that
		\begin{align} \label{eq:estcomp}
			|Df|(\Omega \cap B_{r}(V_r)) - \Psi_{\Omega}(f;B_{r}(V_r)) & \leq |Df_{r}|(\Omega \cap B_{r}(V_r)) \nonumber \\
			& \leq |Df_{r}|(\C_{r} \cap B_{r}(V_r)) \\
			& = \int_{\C_{r} \cap B_{r}(V_r)} |\nabla g_{r} (x)| dx \, . \nonumber
		\end{align}
		Let us now compute the gradient of $g_{r}$. By \eqref{eq:DYr}, setting
		\begin{equation*}
			\nu_{r}(x) = \dfrac{x - V_r}{|x - V_r|} \, , \qquad Y_{r} = Y_{r}(x) \, , 
		\end{equation*}
		we obtain
		\begin{align*}
			\nabla g_{r} (x) = D Y_{r} \cdot \nabla f (Y_r)
			= \dfrac{r}{|x - V_r|} \, \nabla f(Y_{r})^{\nu_{r}(x)^\perp} \, ,
		\end{align*}
		where $\nabla f(Y_r)^{\nu_{r}(x)^\perp}$ denotes the projection of $\nabla f(Y_r)$ onto the hyperplane
		\begin{equation*}
			\nu_r(x)^{\perp} := \{ y \in \R^n \, : \, \langle \nu_r(x) , y \rangle = 0 \} \, .
		\end{equation*}
Going on with the computations, we obtain
\begin{align*}
|\nabla g_{r} (x)| & = \dfrac{r}{|x - V_r|} \, \sqrt[]{|\nabla f (Y_r)|^{2} - \langle \nabla f(Y_r) , \nu_r(x) \rangle^{2} } \\
& = \dfrac{r}{|x - V_r|} \, \left| \nabla f (Y_r) \right| \, \sqrt[]{1 - \dfrac{\left \langle \nabla f (Y_r), \nu_r(x) \right \rangle^{2}}{|\nabla f (Y_r)|^{2}}}  \, . 
\end{align*}
Consequently, we get
\begin{align} \label{eq:estintnormnablafr}
\int_{\C_r \cap B_{r}(V_r)} |\nabla g_r| dx & = \int_{0}^{r} \int_{\C_{r} \cap \de B_{\rho}(V_r)} \dfrac{r}{\rho} \, \left| \nabla f (Y_r) \right| \, \sqrt{1 - \dfrac{\left \langle \nabla f (Y_r), \nu_r \right \rangle^{2}}{|\nabla f (Y_r)|^{2}}} \, d \Hau^{n-1} \, d \rho \nonumber 
\\
			& = \int_{0}^{r} \left( \dfrac{\rho}{r} \right)^{n-2} d \rho \int_{\C_{r} \cap \de B_{r}(V_r)} \left| \nabla f  \right| \, \sqrt{1 - \dfrac{\left \langle \nabla f , \nu_r \right \rangle^{2}}{|\nabla f |^{2}}} \, d \Hau^{n-1} \nonumber \\
			& \leq \dfrac{r}{n-1} \, \Bigg\{ \int_{\Omega \cap \de B_{r}(V_r)} |\nabla f | d \Hau^{n-1} - \dfrac{1}{2} \, \int_{\Omega \cap \de B_{r}(V_r)} \dfrac{\left \langle \nabla f , \nu_r \right \rangle^{2}}{|\nabla f|} \, d \Hau^{n-1} \Bigg\} \, . 
\end{align}
		
		Combining \eqref{eq:estcomp} and \eqref{eq:estintnormnablafr}, we get
		\begin{align} \label{eq:MonotNOcoarea}
			& \dfrac{r}{2(n-1)} \, \int_{\Omega \cap \de B_{r}(V_r)} \dfrac{\left \langle \nabla f , \nu_r \right \rangle^{2}}{|\nabla f|} d \Hau^{n-1}\\
			& \qquad \leq \dfrac{r}{n-1} \, \int_{\Omega \cap \de B_{r}(V_r)} |\nabla f| d \Hau^{n-1} - \int_{\Omega \cap B_{r}(V_r)} |\nabla f| dx + \Psi_{\Omega}(f;B_{r}(V_r)) \, . \nonumber
		\end{align}
		Multiplying both sides of \eqref{eq:MonotNOcoarea} by $(n-1) r^{-n}$ and observing that $r = \phi(y)$ for any $y \in \de B_{r}(V_r)$, we get
\begin{align} \label{eq:Monotcoarea}
\dfrac{1}{2} \, \int_{\Omega \cap \de B_{r}(V_r)}& \left \langle \nabla f , \dfrac{\nabla \phi}{|\nabla \phi|} \right \rangle^{2} \dfrac{1}{|\nabla f | \, \phi^{n-1}} \, d \Hau^{n-1} \nonumber \\
			& \leq \dfrac{1}{r^{n-1}} \, \int_{\Omega \cap \de B_{r}(V_r)} |\nabla f | \, d \Hau^{n-1} + \dfrac{n-1}{r^{n}} \, \int_{\Omega \cap B_{r}(V_r)} |\nabla f| dx + \dfrac{n-1}{r^{n}} \, \Psi_{\Omega}(f;B_{r}(V_r)) \nonumber \\
			&= \dfrac{d}{dr} \left( \dfrac{1}{r^{n-1}} \, \int_{\Omega \cap B_{r}(V_r)} |\nabla f||\nabla \phi| dx \right) \\
			& \qquad\qquad + \dfrac{n-1}{r^{n}} \, \int_{\Omega \cap B_{r}(V_r)} |\nabla f |(|\nabla \phi| - 1) dx + \dfrac{n-1}{r^{n}} \, \Psi_{\Omega}(f;B_{r}(V_r)) \, . \nonumber
		\end{align}
		Let us now integrate \eqref{eq:Monotcoarea} between $0 < r_{1} < r_{2} < R$. We then achieve
		\begin{align} \label{eq:intMonotcoarea}
			& \dfrac{1}{2} \, \int_{\Omega \cap \cA_{r_{1},r_{2}}} \left \langle \nabla f(x) , \dfrac{\nabla \phi(x)}{|\nabla \phi(x)|} \right \rangle^{2} \dfrac{|\nabla \phi(x)|}{|\nabla f (x)| \, \phi(x)^{n-1}} \, dx \nonumber \\
			& \qquad \leq\ \dfrac{1}{r_{2}^{n-1}} \, \int_{\Omega \cap B_{r_{2}}(V_{r_2})} |\nabla f (x)||\nabla \phi(x)| dx - \dfrac{1}{r_{1}^{n-1}} \, \int_{\Omega \cap B_{r_{1}}(V_{r_1})} |\nabla f (x)||\nabla \phi(x)| dx \\
			&\qquad \qquad\qquad +\int_{r_{1}}^{r_{2}} \dfrac{n-1}{r^{n}} \, \int_{\Omega \cap B_{r}(V_r)} |\nabla f (x)|(|\nabla \phi(x)| - 1) dx \, dr + \int_{r_{1}}^{r_{2}} \dfrac{n-1}{r^{n}} \, \Psi_{\Omega}(f;B_{r}(V_r)) dr \, . \nonumber
		\end{align}
		By H\"older's Inequality, we get
		\begin{align} \label{eq:MonHold}
			&\left( \int_{\Omega \cap \cA_{r_{1},r_{2}}} \left| \left\langle \nabla f (x) , \nabla \phi (x) \right\rangle \right| \, \dfrac{dx}{\phi(x)^{n-1}} \right)^{2} \\ 
			& \qquad \leq \left( \int_{\Omega \cap \cA_{r_{1},r_{2}}} \frac{|\nabla \phi(x)|}{\phi(x)^{n-1}} |\nabla f (x)| dx \right) \cdot \left( \int_{\Omega \cap \cA_{r_{1},r_{2}}} \left\langle \nabla f (x) , \dfrac{\nabla \phi (x)}{|\nabla \phi (x)|} \right\rangle^{2} \, \dfrac{|\nabla \phi(x)|\,dx}{|\nabla f (x)| \phi(x)^{n-1}} \right) \, . \nonumber
		\end{align}
		Putting together \eqref{eq:intMonotcoarea} and \eqref{eq:MonHold}, we obtain 
		\begin{align} \label{eq:finalmonotf}
			& \left( \int_{\Omega \cap \cA_{r_{1},r_{2}}} \left| \left\langle \nabla f (x) , \nabla \phi (x) \right\rangle \right| \, \dfrac{dx}{\phi(x)^{n-1}} \right)^{2} \leq 2\left( \int_{\Omega \cap \cA_{r_{1},r_{2}}} \frac{|\nabla \phi(x)|}{\phi(x)^{n-1}} |\nabla f (x)| dx \right) \nonumber \\
			& \qquad \cdot \Bigg{(} \dfrac{1}{r_{2}^{n-1}} \, \int_{\Omega \cap B_{r_{2}}(V_{r_2})} |\nabla f (x)| dx - \dfrac{1}{r_{1}^{n-1}} \, \int_{\Omega \cap B_{r_{1}}(V_{r_1})} |\nabla f (x)| dx \\
			& \qquad \qquad+ \int_{r_{1}}^{r_{2}} \dfrac{n-1}{\rho^{n}} \, \Psi_{\Omega}(f;B_{\rho}(V_{\rho})) d\rho + G(f;r_1,r_{2}) \Bigg{)} \nonumber \, ,
		\end{align}
		where $G(f;r_1,r_2)$ is as in \eqref{eq:errtermfBV}. This proves \eqref{eq:finalmonotfBV} for all $f \in BV(\Omega) \cap C^1(\Omega)$. 
		
		Let now $f \in BV(\Omega)$. We can select a sequence $f_j \in BV(\Omega) \cap C^1(\Omega)$ such that
		\begin{equation} \label{eq:concl1}
			||f_j - f||_{L^1(\Omega)} \to 0 \, , \qquad |Df_j|(\Omega) \to |Df|(\Omega) \, , \qquad Df_j \mathop\weakto^\ast Df \, \, \, \, \, \text{in $\Omega \, .$}
		\end{equation}
		In particular, by the continuity of the trace with respect to the strict convergence, we have
		\begin{equation} \label{eq:concl2}
			||\Tr^{+}(f_j,\de \Omega) - \Tr^{+}(f,\de \Omega)||_{L^1(\de \Omega, \Hau^{n-1})} \to 0 \, .
		\end{equation}
		Let us consider the extensions
		\begin{equation*}
			f_{0,j}(x) = 
			\begin{cases}
				f_j(x) & \text{if $x \in \Omega$} \\
				0 & \text{if $x \in \R^n \setminus \Omega \, ,$}
			\end{cases}
			\qquad
			f_{0}(x) = 
			\begin{cases}
				f(x) & \text{if $x \in \Omega$} \\
				0 & \text{if $x \in \R^n \setminus \Omega \, .$}
			\end{cases}
		\end{equation*}
		We observe that, by \eqref{eq:concl1}, \eqref{eq:concl2},
		\begin{equation*}
			||f_{0,j} - f_{0}||_{L^1(\R^n)} \to 0 \, , \qquad |D f_{0,j}|(\R^n) \to |Df_{0}|(\R^n) \, .
		\end{equation*}
		By Proposition \ref{prop:convBVboost}, for almost all $0 < r < R$,
		\begin{equation} \label{eq:concl3}
			|D f_{0,j}|(B_r(V_r)) \to |D f_{0}|(B_r(V_r)) \, , \qquad ||\Tr^{+}(f_{0,j},\de B_r(V_r)) - \Tr^{+}(f_{0},\de B_r(V_r))||_{L^1(\de B_r(V_r))}\to 0 \, ,
		\end{equation}
		and in particular, owing to \eqref{eq:concl2},
		\begin{equation} \label{eq:concl3bis}
			|D f_j|(\Omega \cap B_r(V_r)) = |D f_{0,j}|(\Omega \cap B_r(V_r)) \to |Df_{0}|(\Omega \cap B_r(V_r)) = |Df|(\Omega \cap B_r(V_r)) \, .
		\end{equation}
		Now \eqref{eq:concl3}, \eqref{eq:concl3bis} allow us to apply Lemma \ref{lem:upbmingap}, and we deduce that
		\begin{equation*}
			|\Psi_\Omega(f_j;B_r(V_r)) - \Psi_\Omega(f;B_r(V_r))| \to 0 \, , \qquad \text{as $j \ra \infty \, ,$}
		\end{equation*}
		for almost all $0<r<R$. This implies that
		\begin{equation*}
			\int_{r_{1}}^{r_{2}} \dfrac{n-1}{\rho^{n}} \, \Psi_{\Omega}(f_j;B_{\rho}(V_{\rho})) d\rho \to \int_{r_{1}}^{r_{2}} \dfrac{n-1}{\rho^{n}} \, \Psi_{\Omega}(f;B_{\rho}(V_{\rho})) d\rho \, .
		\end{equation*}
		Finally, to conclude that the RHS of \eqref{eq:finalmonotf} for $f = f_j$, passes to the limit as $j \ra \infty$, giving precisely the RHS of \eqref{eq:finalmonotfBV}, it suffices to show that the terms 
		\begin{equation*}
			\int_{\Omega \cap \cA_{r_1,r_2}} \phi^{1-n} \, d |Df_j|\,, \quad  \int_{\Omega \cap B_{r_1}(V_{r_1})} (|\nabla \phi|-1) \, d |Df_j|\,, \quad \int_{\Omega \cap B_{r_2}(V_{r_2})} (|\nabla \phi|-1) \, d |Df_j| \,,
		\end{equation*}
converge as $j \ra \infty$ respectively to
		\begin{equation*}
			\int_{\Omega \cap \cA_{r_1,r_2}} \phi^{1-n} \, d |Df| \,,\quad \int_{\Omega \cap B_{r_1}(V_{r_1})} (|\nabla \phi|-1) \, d |Df|\,, \quad \int_{\Omega \cap B_{r_2}(V_{r_2})} (|\nabla \phi|-1) \, d |Df| \, .
		\end{equation*}
To see this, it suffices to construct a suitable partition of each domain, for instance using portions of circular annuli whose boundaries are negligible for $|Df|$ and $|Df_j|$ for all $j \geq 1$, to uniformly approximate each integrand by simple functions (up to removing a small neighborhood of $0$ in the case of the last two integrals). About the LHS of \eqref{eq:finalmonotf}, we observe that \eqref{eq:concl1} implies
		\begin{equation*}
			Df_j \mathop\weakto^\ast Df \qquad \text{in $\Omega \cap \cA_{r_1,r_2} \, .$}
		\end{equation*}
		Now, for $f$ smooth, the LHS of \eqref{eq:finalmonotfBV} and the LHS of \eqref{eq:finalmonotf} coincide. Moreover, we have
		\begin{equation*}
			\int_{\Omega \cap \cA_{r_{1},r_{2}}} \phi(x)^{1-n} \left|\left\langle \nu_f(x), \nabla \phi(x) \right\rangle \right| \, d |Df|(x) = \left| \phi^{1-n} \nabla \phi  \cdot Df_j \right|(\Omega \cap \cA_{r_{1},r_{2}}) \, .
		\end{equation*}
		In particular, \eqref{eq:concl1} implies that
		\begin{equation*}
			\phi^{1-n} \nabla \phi \cdot Df_j\ \mathop\weakto^\ast\ \phi^{1-n} \nabla \phi \cdot Df \, ,
		\end{equation*}
		and well-known properties of the weak-star convergence of Radon measures (see \cite{maggi2012sets}) ensure that
		\begin{equation*}
			\left| \phi^{1-n} \nabla \phi \cdot Df \right|(\Omega \cap \cA_{r_{1},r_{2}}) \leq \liminf_{j \ra \infty} \left| \phi^{1-n} \nabla \phi \cdot Df_j \right|(\Omega \cap \cA_{r_{1},r_{2}}) \, .
\end{equation*}
This implies \eqref{eq:finalmonotfBV} and concludes the proof of the theorem.
\end{proof}

The next corollary, a first important consequence of Theorem \ref{thm:MonotonicityInequality}, states the monotonicity of a suitable function of the radius $r$, which is defined by three terms: the renormalized perimeter $\mu_{E}(r)$, the integral of a renormalized minimality gap $\Psi_{\Omega}(E;B_{r}(V_{r}))$, and the visibility error $G(E,r)$. In particular, when $E$ is an almost-minimizer, the infinitesimality of the second and third terms implies that $\mu_{E}(r)$ is ``almost-increasing'', hence that it admits a finite limit as $r\to 0$. This limit represents the perimeter density of $E$ at $0$, see Remark \ref{rem:perdens-offcentric} below. 
\begin{corollary}[Boundary monotonicity for almost-minimizers]\label{cor:BM}
Let $\Omega \subset \R^n$ satisfy the visibility property. Let $E\subset \Omega$ be a local almost-minimizer, such that $\pr(E,B_{r})>0$ for all $r>0$ and 
\[
\int_{0}^{R}\rho^{-n}\Psi_{\Omega}(E;B_{\rho}(V_{\rho}))\, d\rho < +\infty\,.
\]
Then, using the same notation introduced at the beginning of the section, we can find $R'>0$ such that the function
\[
r\mapsto \mu_{E}(r) + (n-1)\int_{0}^{r}\rho^{-n}\Psi_{\Omega}(E;B_{\rho}(V_{\rho}))\, d\rho + G(E;r)
\]
is non-decreasing on $(0,R')$. Moreover, the two terms
$\int_{0}^{r}\rho^{-n}\Psi_{\Omega}(E;B_{\rho}(V_{\rho}))\, d\rho$ and $G(E;r)$ are infinitesimal as $r\to 0$, hence in particular $\mu_{E}(r)$ is ``almost-monotone'' and the limit 
\[
\theta_{E}(0) := \lim_{r\to 0^{+}}\mu_{E}(r)
\] 
exists and is finite. 
\end{corollary}

\begin{proof}[Proof of Corollary \ref{cor:BM}]
By Lemma \ref{lem:density} and the fact that $B_{r}(V_{r}) \subset B_{r+v(r)}$, we can find constants $C,R>0$ such that
\begin{equation*}
\dfrac{|Df|(\Omega \cap B_{r}(V_{r}))}{r^{n-1}} \leq \dfrac{|Df|(\Omega \cap B_{r + v(r)})}{(r + v(r))^{n-1}} \left( 1 + \dfrac{v(r)}{r} \right)^{n-1} \leq C\,,\qquad \text{for all $0<r<R$.}
\end{equation*}
By combining Proposition \ref{prop:Gbound1} with \eqref{eq:Dphi-1} and the previous bound, up to redefining the constants $C,R>0$, we obtain 
\begin{equation}\label{eq:stimaG1}
|G(E;r)| \le C\left(\int_{0}^{r}\gamma_{v}(\rho)\, d\rho + r\gamma_{v}(r)\right)\,,\qquad \text{for all $0<r<R$.}
\end{equation}
Finally, the proof of the corollary follows directly from Theorem \ref{thm:MonotonicityInequality} and from the observation that the RHS of \eqref{eq:stimaG1} is infinitesimal as $r\to 0^{+}$. 
\end{proof}

\begin{remark}\label{rem:perdens-offcentric}
It is easy to check that, under the assumptions of Corollary \ref{cor:BM}, one has
\begin{equation*}
\exists\, \lim_{r\to 0^{+}}\frac{\pr(E;B_{r})}{r^{n-1}} = \theta_{E}(0)\,.
\end{equation*}
Indeed, this is an immediate consequence of the inclusions
\[
B_{r-v(r)}(V_{r}) \subset B_{r} \subset B_{r+v(r)}(V_{r})
\]
combined with $v(r) = o(r)$ as $r\to 0$.
\end{remark}

\section{Blow-up limits of almost-minimizers are cones}
We now apply Theorem \ref{thm:MonotonicityInequality} and prove that any blow-up limit of a local almost-minimizer $E$ of $\pr$ is a perimeter-minimizing cone. 
\begin{theorem} \label{thm:conblowup}
Let $\Omega \subset \R^{n}$ satisfy the visibility property. Let $E \subset \Omega$ be a local almost-minimizer in $\Omega$ such that
\begin{equation*} 
\int_{0}^{R} \dfrac{\Psi_{\Omega}(E;B_r)}{r^n} dr < \infty \, .
\end{equation*}
Fix a decreasing sequence $t_{j} \to 0$ and set $E_{t_{j}} = t_{j}^{-1}E$. Then, up to subsequences, $E_{t_{j}}$ converges to $E_{0} \subset \Omega_{0}$ in $L^{1}_{\loc}(\R^{n})$. Moreover, $E_{0}$ is a nontrivial cone minimizing the relative perimeter in $\Omega_{0}$.
\end{theorem}
\begin{proof}
Set $E_{j}=E_{t_{j}}$ and $\Omega_{j} = \Omega_{t_{j}}$ for more simplicity. Then by the upper density estimate on the relative perimeter of $E$ (Lemma \ref{lem:density}) coupled with analogous estimates satisfied by $\Omega$ Lipschitz, we can find a constant $C>0$ such that, for every fixed $R>0$, 
\begin{align*}
\p(E_{j}; B_{R}) &\le \p(\Omega_{j}; B_{R}) + \p_{\Omega_{j}}(E_{j};B_{R})\\
&= t_{j}^{1-n}\big(\p(\Omega;B_{Rt_{j}}) + \pr(E;B_{Rt_{j}})\big)\\
&\le CR^{n-1}\,.
\end{align*}
By the compactness property of sequences of sets with uniformly bounded relative perimeter the ball $B_{R}$, we conclude that there exists a not relabeled subsequence $E_{j}$ and a set $E_{0}$ of finite perimeter in $B_{R}$, such that $E_{j}\to E_{0}$ in $L^{1}(B_{R})$ as $j\to \infty$. The fact that $E_{0} \subset \Omega_{0}$ up to null sets is immediate, since $E_{j} \subset \Omega_{j}$, for all $j$, and the sequence $\Omega_{j}$ converges to the tangent cone $\Omega_{0}$ locally in Hausdorff distance (hence, in $L^{1}_{\loc}(\R^{n})$) thanks to Proposition \ref{prop:tangcone}. Up to a standard diagonal argument we can assume that the subsequence $E_{j}$ converges to $E_{0}$ in $L^{1}_{\loc}(\R^{n})$. Moreover by the lower-density estimates on the volume of $E$ we also deduce that $E_{0}$ can be neither the empty set, nor the whole $\Omega_{0}$ up to null sets (that is, $E_{0}$ is nontrivial).

By the scaling properties of the perimeter, for any fixed $R>0$ we have
\begin{align*}
\Psi_{\Omega_{t_j}}(E_{t_{j}};B_{R}) = \dfrac{1}{t_{j}^{n-1}} \Psi_{\Omega}(E;B_{t_{j} R}) \leq
\omega_{n}^{1 - \frac{1}{n}} \, R^{n-1} \psi_{\Omega}(E; 0, t_{j} R) \longrightarrow 0 \, , 
\end{align*}
therefore we can apply Lemma \ref{lem:gaplsc} and deduce that 
\[
\Psi_{\Omega_{0}}(E_{0};B_{R}) \le \liminf_{j}\Psi_{\Omega_{t_{j}}}(E_{t_{j}};B_{R}) = 0
\] 
for all $R>0$ and, also owing to Corollary \ref{cor:BM},
\begin{align*}
\p_{\Omega_{0}}(E_{0};B_{R}) &= \lim_{j}\p_{\Omega_{t_{j}}}(E_{t_{j}};B_{R}) = \lim_{j}t_{j}^{1-n}\pr(E;B_{t_{j}R}\cap \Omega) = R^{n-1}\theta_{E}(0)\,.
\end{align*}
Thus, $E_{0}$ is a minimizer for the relative perimeter in the cone $\Omega_{0}$, such that 
\[
\frac{\p(E_{0};B_{R}\cap \Omega_{0})}{R^{n-1}} = \theta_{E}(0)\qquad \text{for all $R>0$.}
\]
Now, the monotonicity inequality \eqref{eq:finalmonotfBV} written for $f = \ch_{E_{0}}$ and $\Omega=\Omega_{0}$ takes the form
\begin{align*}
\Bigg( \int_{\Omega_{0} \cap (B_{r_{2}} \setminus B_{r_{1}})}& \dfrac{| \left\langle \nu_{E_{0}}(x) , x \right\rangle|}{|x|^{n}} \, d |D\ch_{E}|(x) \Bigg)^{2}\\
& \leq \left( \int_{\Omega_{0} \cap (B_{r_{2}} \setminus B_{r_{1}})} |x|^{1-n} \, d |D\ch_{E_{0}}(x)| \right) \cdot \left( \dfrac{\p_{\Omega_{0}}(E_{0}; B_{r_{2}})}{r_{2}^{n-1}}  - \dfrac{\p_{\Omega_{0}}(E_{0}; B_{r_{1}})}{r_{1}^{n-1}} \right) = 0 \, ,
\end{align*}
for almost all $0 < r_{1} < r_{2}$. The only possibility is then that $\langle \nu_{E_{0}}(x), x \rangle = 0$ at $\Hau^{n-1}$-almost every $x \in \de^{\ast} E_{0}$. By \cite[Proposition 28.8]{maggi2012sets} we infer that $E_{0}$ is a cone with vertex at the origin, up to negligible sets, and the proof is concluded.
\end{proof}


\end{document}